\pdfoutput=1
\documentclass[12pt, reqno]{amsart}
\usepackage[utf8]{inputenc}
\usepackage[english]{babel}
\usepackage{braket}
\usepackage{amsmath}
\usepackage{amstext}
\usepackage{amssymb}
\usepackage{amsthm}
\usepackage{mathtools}
\usepackage{stmaryrd}
\usepackage{mathrsfs}
\usepackage{extarrows}
\usepackage{afterpage}
\usepackage{tabu}
\usepackage{array}
\usepackage{microtype}
\usepackage{graphicx}
\usepackage{amsfonts}
\usepackage{units}
\usepackage[shortlabels]{enumitem}
\usepackage{url}
\usepackage{bm}
\usepackage{soul}
\usepackage{tikz}
\usetikzlibrary{matrix,arrows,calc,fit,cd,positioning,intersections,arrows.meta}
\usepackage{hyperref}
\hypersetup{
	colorlinks = true,
	linkcolor = {blue},
	urlcolor = {blue},
	citecolor = {blue}
}
\usepackage{cleveref}
\usepackage{colortbl}

\newtheorem{theorem}{Theorem}[section]
\newtheorem{lemma}[theorem]{Lemma}

\newtheorem{corollary}[theorem]{Corollary}

\theoremstyle{definition}
\newtheorem{definition}[theorem]{Definition}
\newtheorem{remark}[theorem]{Remark}
\newtheorem{example}[theorem]{Example}

\newcommand{\N}{\mathbb{N}}
\newcommand{\Z}{\mathbb{Z}}
\newcommand{\R}{\mathbb{R}}

\newcommand{\Q}{\mathbb{Q}}

\newcommand{\F}{\mathbb{F}}
\renewcommand{\H}{\mathbb{H}}
\renewcommand{\P}{\mathbb{P}}

\newcommand{\W}{\mathcal{W}}

\DeclareMathOperator{\id}{id}

\DeclareMathOperator{\im}{im}

\newcommand{\<}{\langle}
\renewcommand{\>}{\rangle}

\DeclareMathOperator{\Mov}{\textsc{Mov}}
\DeclareMathOperator{\Fix}{\textsc{Fix}}

\DeclareMathOperator{\GL}{GL}

\newenvironment{mymatrix}[1]{
    \left(
    
    \begin{array}{#1}
}
{
    \end{array}
    \right)
}

\allowdisplaybreaks %

\title[Factoring isometries of quadratic spaces]{Factoring isometries of
  quadratic spaces into reflections}
\date{\today}
\author{Jon McCammond}
\author{Giovanni Paolini}

\begin{document}

\begin{abstract}
  Let $V$ be a vector space endowed with a non-degenerate quadratic
  form $Q$.  If the base field $\F$ is different from $\F_2$, it is
  known that every isometry can be written as a product of
  reflections.  In this article, we detail the structure of the poset
  of all minimal length reflection factorizations of an isometry.  If
  $\F$ is an ordered field, we also study
  factorizations into positive reflections, i.e., reflections defined
  by vectors of positive norm.  We characterize such factorizations,
  under the hypothesis that the squares of $\F$ are dense in the
  positive elements (this includes Archimedean and Euclidean fields).
  In particular, we show that an isometry is a
  product of positive reflections if and only if its spinor norm is
  positive.  As a final application, we explicitly describe the poset
  of all factorizations of isometries of the hyperbolic space.
\end{abstract}

\maketitle

Let $V$ be a finite-dimensional vector space over a field $\F$.  A
quadratic form on $V$ is a map $Q \colon V \to \F$ such that: (1) $Q(av) =
a^2 Q(v)$ for all $a \in \F$ and $v \in V$; (2) the polar form $\beta(u,v)
= Q(u+v) - Q(u) - Q(v)$ is bilinear.  When $\F$ is different from the
two-element field $\F_2$, every isometry of a non-degenerate quadratic
space $(V, Q)$ can be written as a product of at most $\dim V$
reflections and the minimal length of a reflection factorization is
determined by geometric attributes of the isometry
\cite{cartan1938theorie, dieudonne1948groupes,
  scherk1950decomposition, dieudonne1955generateurs,
  callan1976generation, taylor1992geometry}.  For some applications,
e.g.\ when studying dual Coxeter systems and the associated Artin
groups \cite{bessis2003dual, brady2008non, mccammond2015dual,
  mccammond2017artin, paolini2019proof}, more fine-grained information
is useful: What is the set of all minimal length reflection
factorizations? What is the combinatorial structure of the intervals
in the orthogonal group $O(V, Q)$, with respect to the metric defined
by the reflection length?  Answers to these questions have been given
for anisotropic quadratic spaces \cite{brady2002partial} and for
(affine) Euclidean spaces \cite{brady2015factoring}.  In the first
part of this paper, we give answers for general quadratic spaces.  Our
treatment is based on Wall's parametrization of the orthogonal group
\cite{wall1959structure, wall1963conjugacy}, which we recall in \Cref{sec:wall-parametrization}.

In the second part of this paper, we turn our attention to the case
where $\F$ is an ordered field.  We say that a reflection with respect
to some vector $v \in V$ is positive if $Q(v)$ is positive.  One can
ask all the previous questions while restricting to factorizations
into positive reflections only.  The following are our main
motivations for studying this problem: (1) understand reflection
factorizations in Coxeter groups (which are discrete groups generated
by positive reflections with respect to some quadratic form in
$\R^n$); (2) describe reflection factorizations of isometries of the
hyperbolic space $\H^n$.  We characterize the positive reflection
length of all isometries, and we describe the minimal factorizations,
under the hypothesis that $\F$ is \emph{square-dense}: the squares of
$\F$ are dense in the set of positive elements.  Most notably, the class of square-dense
ordered fields includes all Archimedean fields (i.e., the subfields of
$\R$) and Euclidean fields (i.e., fields where every positive
element is a square).  In particular, we show that an isometry can be
written as a product of positive reflections if and only if its spinor
norm is positive.

As an application, we study the reflection factorizations of
isometries of the hyperbolic space $\H^n$.  In the hyperboloid model
$\H^n \subseteq \R^{n+1}$, the isometries of $\H^n$ form an index-two
subgroup of the orthogonal group $O(\R^{n+1}, Q)$, where $Q$ is a
quadratic form of signature $(n, 1)$.  In fact, they are precisely the
isometries of $(\R^{n+1}, Q)$ with a positive spinor norm.  This
observation allows us to give an explicit description of minimal
reflection factorizations and intervals in $O(\H^n)$.

\section{Wall's parametrization of the orthogonal group}
\label{sec:wall-parametrization}

In this section, we recall Wall's parametrization of the orthogonal
group of a quadratic space, which was first introduced in
\cite{wall1959structure}.  To be as self-contained as
possible, we give proofs for the most important results.
We largely follow the treatment of \cite[Chapter 11]{taylor1992geometry}, but the
reader can also refer to \cite{wall1959structure, wall1963conjugacy,
  hahn1979unipotent}.

Let $V$ be a finite-dimensional vector space over a field $\F$.  For
now, no hypothesis on $\F$ is required.  A \emph{quadratic form} on
$V$ is a map $Q \colon V \to \F$ such that:
\begin{enumerate}
    \item $Q(av) = a^2 Q(v)$ for all $a \in \F$ and $v \in V$;
    \item the map $\beta(u,v) = Q(u+v) - Q(u) - Q(v)$ is bilinear.
\end{enumerate}
The pair $(V, Q)$ is called a \emph{quadratic space}, and the symmetric bilinear form $\beta$ is called the \emph{polar form} of
$Q$.
From now on, assume that $(V, Q)$ is a \emph{non-degenerate} quadratic space, i.e., the polar form
$\beta$ is non-degenerate: $\beta(u, v) = 0$ for all $v \in V$ implies
$u = 0$.

If the characteristic of $\F$ is not $2$, the polar form $\beta$
determines $Q$ via the relation $Q(u) = \frac12 \beta(u, u)$.  On the
other hand, if the characteristic of $\F$ is $2$, $\beta$ is
alternating (i.e., $\beta(u, u) = 0$ for all $u \in V$) and does not
determine $Q$.

A non-zero vector $u \in V$ is \emph{isotropic} if
$\beta(u, u) = 0$ and it is \emph{singular} if $Q(u) = 0$.
These two notions coincide when the characteristic of $\F$ is not $2$.
Given a linear subspace $W \subseteq V$, its orthogonal subspace is
defined as $W^\perp = \{ v \in V \mid \beta(v, w) = 0 \text{ for all }
w \in W \}$.
A subspace $W \subseteq V$ is \emph{totally singular} if $Q(u) = 0$ for
all $u \in W$, and it is \emph{non-degenerate} if $W \cap W^\perp =
\{0\}$ (i.e., if $\beta|_W$ is non-degenerate).  Since $\beta$ is
non-degenerate, we have that $\dim(W) + \dim(W^\perp) = \dim(V)$ and
$(W^\perp)^\perp = W$ for every subspace $W \subseteq V$.  However,
note that $W \cap W^\perp$ might be non-trivial, so $V$ is not
necessarily the direct sum of $W$ and $W^\perp$.  If $V = W_1 \oplus
W_2$ and $W_1 = W_2^\perp$, we also write $V = W_1 \perp W_2$.

\begin{definition}[Orthogonal group]
  The \emph{orthogonal group} of $(V, Q)$ is \[ O(V, Q) = \{ f \in
  \GL(V) \mid Q(f(u)) = Q(u) \, \text{ for all $u \in V$} \}. \] The
  elements of the orthogonal group are called \emph{isometries}.  We
  also write $O(V)$ in place of $O(V, Q)$, since the ambient quadratic
  form $Q$ is always fixed.
\end{definition}

By definition, an isometry $f \in O(V)$ also preserves the polar form
$\beta$:
\begin{align*}
  \beta(f(u), f(v)) &= Q(f(u) + f(v)) - Q(f(u)) - Q(f(v)) \\
  &= Q(f(u+v)) - Q(f(u)) - Q(f(v)) \\
  &= Q(u+v) - Q(u) - Q(v) \\
  &= \beta(u, v).
\end{align*}
Notice that if $f \colon V \to V$ is a linear map that preserves $\beta$, then $f \in \GL(V)$ because $\beta$ is non-degenerate.

Our aim is to characterize the factorizations of isometries as
products of reflections.  A \emph{reflection} is a non-trivial
isometry that fixes every vector in a hyperplane of $V$.  Every
reflection can be written as
\begin{equation}
r_v(u) = u - \frac{\beta(u, v)}{Q(v)} v
\label{eq:reflection}
\end{equation}
for some non-singular vector $v \in V$, and $r_v$ is called the
reflection with respect to $v$.  Note that $r_v = r_{w}$ for every
non-zero scalar multiple $w$ of $v$.  In addition, $r_v$ fixes the
hyperplane $\<v\>^\perp$, sends $v$ to $-v$, has order $2$ and
determinant $-1$.  The set of reflections is closed under conjugation:
$f r_v f^{-1} = r_{f(v)}$ for every $f \in O(V)$.

The following are two important subspaces associated with an isometry.

\begin{definition}
  Given an isometry $f \in O(V)$, its \emph{fixed space} is $\Fix(f) =
  \ker(\id - f)$ and its \emph{moved space} is $\Mov(f) = \im(\id -
  f)$.
\end{definition}

The fixed space is simply the subspace of vectors that are fixed by
$f$. The moved space is the subspace of ``movement'' vectors $f(u)-u$,
for $u \in V$.  It is also called the \emph{residual space} of $f$.
The notation ``$\Fix(f)$'' and ``$\Mov(f)$'' is the one used in
\cite{brady2015factoring}, but several different notations for the
moved space have appeared in the literature, including $V_f$, $[V, f]$, and $M(f)$
\cite{wall1959structure, wall1963conjugacy,taylor1992geometry,brady2002partial}.

\begin{lemma}\label{lemma:fix-move-orthogonal}
  For every isometry $f \in O(V)$, we have that $\Fix(f) =
  \Mov(f)^\perp$.
\end{lemma}

\begin{proof}
  By definition, the subspaces $\Fix(f)$ and $\Mov(f)$ have
  complementary dimensions, so it is enough to show that $\beta(u, v)
  = 0$ for every $u \in \Fix(f)$ and $v \in \Mov(f)$.  For this, write
  $v = w - f(w)$ for some $w \in V$.  Then
  \begin{align*}
    \beta(u, v) &= \beta(u, w - f(w))
    = \beta(u, w) - \beta(u, f(w)) \\
    &= \beta(u, w) - \beta(f(u), f(w))
    = 0. \qedhere
  \end{align*}
\end{proof}

Notice that an isometry $f \in O(V)$ is a reflection if and only
if $\Mov(f)$ is one-dimensional (in which case $f=r_v$ where $\Mov(f) = \< v \>$), and this happens if and only if $\Fix(f)$ is a hyperplane (in
which case $\Fix(f) = \< v \>^\perp$).

When $f$ is not a reflection, its moved space $\Mov(f)$ does not
determine $f$ uniquely.  For example, if $V = \R^n$ and $Q$ is the standard
(positive definite) quadratic form, a $2$-dimensional subspace $W
\subseteq V$ is the moved space of infinitely many rotations.  By
\Cref{lemma:fix-move-orthogonal}, each of $\Fix(f)$ and $\Mov(f)$
determines the other, so no additional information comes from knowing
both of them.
The Wall form adds the information needed to determine $f$.

\begin{definition}[\cite{wall1959structure}]\label{def:wall-form}
  Let $f \in O(V)$ be an isometry.  The \emph{Wall form} of $f$ is the
  bilinear form $\chi_f$ on $\Mov(f)$ defined as $\chi_f(u, v) =
  \beta(w, v)$, where $w \in V$ is any vector such that $u = w -
  f(w)$.
\end{definition}

\begin{theorem}\label{thm:wall-form}
  The Wall form $\chi_f$ is a well-defined non-degenerate bilinear
  form on $\Mov(f)$, and it satisfies $\chi_f(u,u) = Q(u)$ for all $u
  \in V$.
\end{theorem}

\begin{proof}
  Suppose that $u = w - f(w) = w' - f(w')$ for some $w, w' \in V$.
  Then $w - w' \in \Fix(f) = \Mov(f)^\perp$ by
  \Cref{lemma:fix-move-orthogonal}, and therefore $\beta(w, v) - \beta(w', v) = \beta(w - w', v) = 0$,
  so $\chi_f(u, v)$ is well-defined.
	
  It is immediate to see that $\chi_f$ is a bilinear form.  If
  $\chi_f$ is degenerate, then there is a non-zero vector $v \in
  \Mov(f)$ such that $\chi_f(u, v) = 0$ for all $u \in \Mov(f)$.  Then
  $\beta(w, v) = 0$ for all $w \in V$.  This is impossible, because
  $\beta$ is non-degenerate.
	
  Finally, if $u = w - f(w)$, we have $\chi_f(u, u) = \beta(w, u) =
  -\beta(w, -u) = Q(w) + Q(u) - Q(w - u) = Q(w) + Q(u) - Q(f(w)) =
  Q(u)$.
\end{proof}

The Wall form $\chi_f$ is not necessarily symmetric.  In fact, we show
in \Cref{lemma:wall-form-properties} that $\chi_f$ is symmetric if and
only if $f$ is an involution.  As anticipated, the Wall form $\chi_f$
carries enough information to recover the isometry $f$.

\begin{theorem}[Wall's parametrization]
  \label{thm:wall-parametrization}
  The map $f \mapsto (\Mov(f), \chi_f)$ is a one-to-one correspondence
  between the orthogonal group $O(V)$ and the set of pairs $(W, \chi)$
  such that $W$ is a subspace of $V$ and $\chi$ is a non-degenerate
  bilinear form on $W$ satisfying $\chi(u,u) = Q(u)$ for $u \in W$.
\end{theorem}

\begin{proof}
  To prove injectivity, consider two isometries $f, g \in O(V)$ such
  that $\Mov(f) = \Mov(g) = W$ and $\chi_f = \chi_g = \chi$.  By
  definition of Wall form, $\chi_f(w-f(w), v) = \beta(w, v) =
  \chi_g(w-g(w), v)$ and therefore $\chi(w-f(w), v) = \chi(w-g(w),
  v)$, for every $v \in W$ and $w \in V$.  Since $\chi$ is
  non-degenerate, this implies that $w-f(w) = w-g(w)$ for all $w \in
  V$, thus $f = g$.
	
  To prove surjectivity, given a pair $(W, \chi)$, we want to construct
  an isometry $f \in O(V)$ such that $\Mov(f) = W$ and $\chi_f =
  \chi$.  For $w \in V$, denote by $\alpha_w \in W^*$ the linear
  functional given by $\alpha_w(v) = \beta(w, v)$.  Since $\chi$ is
  non-degenerate, the linear map $\varphi \colon W \to W^*$ given by
  $\varphi(u)(v) = \chi(u, v)$ is an isomorphism.  Define $f\colon V
  \to V$ as follows: $f(w) = w - \varphi^{-1}(\alpha_w)$.  By
  construction, for any $w \in V$ and $v \in W$ we have
  \begin{equation}
    \beta(w, v) = \alpha_w(v) = \varphi(w - f(w))(v) = \chi(w - f(w), v).
    \label{eq:wall-parametrization}
  \end{equation}
	
  This allows us to check that $f$ is an isometry. Indeed, by setting
  $v = w - f(w)$ in \cref{eq:wall-parametrization} we obtain
  \begin{align*}
  	\beta(w, w - f(w)) &= \chi(w - f(w), w - f(w)) = Q(w-f(w)) \\
  	&= Q(w) + Q(f(w)) - \beta(w, f(w)),
  \end{align*}
  which simplifies to $Q(f(w)) = Q(w)$.  By
  definition of $f$, we immediately see that $\Mov(f) = W$, and
  \cref{eq:wall-parametrization} implies that $\chi = \chi_f$.
\end{proof}

We now list some properties of the Wall form.

\begin{lemma}\label{lemma:wall-form-properties}
  For every $f \in O(V)$ and $u, v \in \Mov(f)$, the following properties hold.
  \begin{enumerate}[(i)]
  \item $\chi_f(u, v) + \chi_f(v, u) = \beta(u, v)$.
  \item $\chi_f(f(u), v) = -\chi_f(v, u)$.
  \item $\Mov(f) = \Mov(f^{-1})$ and $\chi_{f^{-1}}(u, v) = \chi_f(v,
    u)$.
  \item $\Mov(gfg^{-1}) = g(\Mov(f))$ and $\chi_{gfg^{-1}}(g(u), g(v))
    = \chi_f(u, v)$ for every $g \in O(V)$.
  \item $\chi_f$ is symmetric if and only if $f$ is an involution.
  \end{enumerate}
\end{lemma}

\begin{proof}
  (i) follows from the identity $\chi_f(u,u) = Q(u)$ of
  \Cref{thm:wall-form}, by replacing $u$ with $u+v$.  To prove (ii),
  write
  \begin{align*}
      \chi_f(f(u), v) + \chi_f(v, u) &= \chi_f(f(u), v) -
  \chi_f(u, v) + \beta(u, v) \\
  &= \beta(u, v) - \chi_f(u-f(u), v) = 0,
  \end{align*}
  where the first equality follows from (i) and the last equality
  follows from the definition of $\chi_f$.  In (iii), it is obvious
  that $\Mov(f) = \Mov(f^{-1})$, and we only need to check that
  $\tilde\chi_{f^{-1}}(u,v) := \chi_f(v,u)$ satisfies
  \Cref{def:wall-form} for $f^{-1}$.  Indeed,
  \begin{align*}
    \tilde\chi_{f^{-1}}\big(w-f^{-1}(w), v \big) &= \chi_f\big(v, w-f^{-1}(w)\big) = - \chi_f(f(w) - w, v) \\
    &= \chi_f(w - f(w), v) =
  \beta(w, v),
  \end{align*}
  where the second equality follows from (ii).  In
  (iv), it is immediate that $\Mov(gfg^{-1}) = g(\Mov(f))$. We show
  that $\tilde\chi_f (u,v) := \chi_{gfg^{-1}}(g(u), g(v))$ satisfies
  \Cref{def:wall-form}:
  \begin{align*}
    \tilde\chi_f(w - f(w), v) &=
    \chi_{gfg^{-1}}(g(w - f(w)), g(v)) \\
    &= \chi_{gfg^{-1}}(g(w) - g(f(w)), g(v)) \\
    &= \beta(g(w), g(v)) = \beta(w, v).
  \end{align*}
  Therefore $\tilde\chi_f
  = \chi_f$.  In (v), if $\chi_f$ is symmetric, then we can apply
  (iii) and deduce that $\Mov(f^{-1}) = \Mov(f)$ and
  $\chi_{f^{-1}}(u,v) = \chi_f(v, u) = \chi_f(u, v)$.  Then
  $\chi_{f^{-1}} = \chi_f$, so $f^{-1} = f$ by
  \Cref{thm:wall-parametrization}.  Conversely, if $f$ is an
  involution, then $\chi_f(u, v) = \chi_{f^{-1}}(u, v) = \chi_f(v,
  u)$, so $\chi_f$ is symmetric.
\end{proof}

Fix a subspace $W \subseteq V$, and look at all isometries $f \in
O(V)$ such that $\Mov(f) = W$.  Property (i) of
\Cref{lemma:wall-form-properties} says that the symmetrization of the
Wall form $\chi_f$ is necessarily equal to the ambient bilinear form
$\beta$ (restricted to $W = \Mov(f)$).  In particular, if $W$ is
non-degenerate and the characteristic of $\F$ is not $2$, there is
exactly one isometry $f$ such that $\Mov(f) = W$ and $\chi_f$ is
symmetric, and $f$ is an involution by property (v).  On the opposite
side, if $W$ is totally singular, then $\chi_f$ is alternating by
\Cref{thm:wall-form}.  In this case, isometries $f$ with $\Mov(f) = W$
only exist if $\dim W$ is even (otherwise every alternating bilinear
form on $W$ is degenerate, as the rank is necessarily even; see for
example \cite[Theorem 2.10]{grove2002classical}).

\section{Factorizations and reflection length}\label{sec:factorizations}

In this section, we continue to follow \cite{wall1959structure} and
\cite[Chapter 11]{taylor1992geometry} and show how Wall's
parametrization leads to a nice procedure to build factorizations of
isometries.  For a field $\F \neq \F_2$, this allows proving that any
isometry $f \in O(V)$ can be written as a product of reflections. It
also allows us to characterize the \emph{reflection length}, i.e., the
minimal length $k$ of a factorization $f = r_1r_2\dotsm r_k$ as a
product of reflections.  We refer to \cite[Theorem
  11.41]{taylor1992geometry} for the case $\F = \F_2$, which we do not
treat here.  Finally, at the end of this section, we introduce the
spinor norm.

\begin{definition}[Orthogonal complements]\label{def:left-right-orthogonal-complement}
  Let $\chi$ be a non-degenerate bilinear form on a finite-dimensional
  vector space $W$.  Define the \emph{left} and \emph{right orthogonal
    complement} of a subspace $U \subseteq W$ as \begin{align*}
    U^{\triangleleft} &= \{ v \in W \mid \chi(v, u) = 0 \, \text{ for
      all } u \in U \} \\ U^{\triangleright} &= \{ v \in W \mid
    \chi(u, v) = 0 \, \text{ for all } u \in U \},
	\end{align*} respectively.
\end{definition}

Since $\chi$ is non-degenerate, we have that $\dim U^\triangleleft =
\dim U^\triangleright = \dim W - \dim U$. As an immediate consequence,
$(U^\triangleright)^\triangleleft = (U^\triangleleft)^\triangleright =
U$.  We will mostly use this notation in the case where $\chi =
\chi_f$ is the Wall form of an isometry $f \in O(V)$ and $W =
\Mov(f)$.

The following is the basic building block that allows us to construct
factorizations of isometries.

\begin{theorem}[Factorization theorem]\label{thm:factorization}
  Let $f \in O(V)$ be an isometry, and let $U_1 \subseteq \Mov(f)$ be
  a subspace such that the restriction $\chi_1 = \chi_f|_{U_1}$ is
  non-degenerate.  Let $U_2 = U_1^\triangleright$ (respectively, $U_2
  = U_1^\triangleleft$), and $\chi_2 = \chi_f|_{U_2}$.  Denote by
  $f_1$ and $f_2$ the elements of $O(V)$ associated with $(U_1,
  \chi_1)$ and $(U_2, \chi_2)$ under Wall's parametrization.
  \begin{enumerate}[(a)]
  \item $\Mov(f) = U_1 \oplus U_2$, and $f = f_1 f_2$ (respectively,
    $f = f_2f_1$).
  \item $f_1 f_2 = f_2 f_1$ if and only if $\Mov(f) = U_1 \perp U_2$.
    In this case, $f_1$ coincides with $f$ on $U_2^\perp$, and $f_2$
    coincides with $f$ on $U_1^\perp$.
  \end{enumerate}
  Conversely, every factorization $f = f_1f_2$ with $\Mov(f) =
  \Mov(f_1) \oplus \Mov(f_2)$ arises in this way.
\end{theorem}

\begin{proof}
  We prove part (a) in the case $U_2 = U_1^\triangleright$, the case
  $U_2 = U_1^\triangleleft$ being analogous.  Since $\chi_1$ is
  non-degenerate, no non-zero vector of $U_1$ can be right-orthogonal
  to all of $U_1$.  This means that $U_1 \cap U_2 = \{0\}$.  We also
  have $\dim U_1 + \dim U_2 = \dim \Mov(f)$, and therefore $\Mov(f) =
  U_1 \oplus U_2$.
	
  Notice that $\chi_2$ is non-degenerate because $\chi_f$ is
  non-degenerate, so $f_2$ is well-defined.  To prove that $f =
  f_1f_2$, consider the following chain of equalities that holds for
  every $w \in V$, $u_1 \in U_1$, and $u_2 \in U_2$:
  \begin{align*}
    \chi_f &\big(w - f_1f_2(w), u_1 + u_2\big) = \chi_f\big(w - f_2(w)
    + f_2(w) - f_1f_2(w), u_1 + u_2\big) \\ &= \chi_f(w - f_2(w), u_1
    + u_2) + \chi_f\big((\id - f_1)f_2(w), u_1 + u_2\big)
    \\ &\stackrel{(1)}{=} \chi_f\big(w-f_2(w), u_1\big) + \chi_f\big(w
    - f_2(w), u_2\big) + \chi_f\big((\id - f_1)f_2(w), u_1 \big)
    \\ &\stackrel{(2)}{=} \beta\big(w - f_2(w), u_1\big) + \beta(w,
    u_2) + \beta\big(f_2(w), u_1\big) \\ &= \beta(w, u_1 + u_2) \\ &=
    \chi_f\big(w - f(w), u_1 + u_2\big).
  \end{align*}
  Here (1) follows from bilinearity of $\chi_f$, the term
  $\chi_f\big((\id-f_1)f_2(w), u_2\big)$ vanishing because
  $(\id-f_1)f_2(w) \in \Mov(f_1) = U_1$ and $u_2 \in U_2$; in (2), the
  first term is rewritten using property (i) of
  \Cref{lemma:wall-form-properties}, whereas the other two terms are
  rewritten using the definitions of $\chi_1$ and $\chi_2$.  From the
  previous equalities and the fact that $\chi_f$ is non-degenerate, it
  follows that $w - f_1f_2(w) = w - f(w)$ for all $w \in V$, so $f =
  f_1f_2$.
	
  We now prove part (b).  Suppose that $f_1f_2 = f_2f_1$.  By
  property (iv) of \Cref{lemma:wall-form-properties}, $f$ fixes
  $\Mov(f_1) = U_1$.  Then, by property (ii), we have that $\chi_f(u_2, u_1) = - \chi_f(f(u_1), u_2) = 0$ for all $u_1 \in U_1$ and $u_2 \in U_2$.
  Therefore $U_2 = U_1^\triangleright = U_1^\triangleleft$.
  Property (i) implies that $\Mov(f) = U_1 \perp U_2$.
	
  Conversely, suppose that $\Mov(f) = U_1 \perp U_2$.  Since $U_2 =
  U_1^\triangleright$, property (i) of
  \Cref{lemma:wall-form-properties} implies that $U_2 =
  U_1^\triangleleft$.  By the first part of this theorem, we obtain
  that $f = f_2f_1$, and therefore $f_1f_2 = f_2f_1$.  In addition,
  $\Fix(f_2) = \Mov(f_2)^\perp = U_2^\perp$, and thus $f(v) =
  f_1f_2(v) = f_1(v)$ for every $v \in U_2^\perp$. Similarly, $f(v) =
  f_2f_1(v) = f_2(v)$ for every $v \in U_1^\perp$.
	
  Finally, given any factorization $f = f_1f_2$ such that $\Mov(f) =
  \Mov(f_1) \oplus \Mov(f_2)$, we need to show that
  $\chi_f|_{\Mov(f_2)} = \chi_{f_2}$.  Let $u, v \in \Mov(f_2)$.  By
  definition of $\chi_f$, we have that $\chi_f(u, v) = \beta(w, v)$,
  where $w \in V$ is a vector such that $u = w - f(w)$.  Now write $u
  = w - f_2(w) + f_2(w) - f_1f_2(w)$, and notice that $w-f_2(w) \in
  \Mov(f_2)$ and $f_2(w) - f_1f_2(w) \in \Mov(f_1)$.  Since $u \in
  \Mov(f_2)$ and $\Mov(f) = \Mov(f_1) \oplus \Mov(f_2)$, we have that
  $u = w - f_2(w)$.  Then $\chi_{f_2}(u, v) = \beta(w, v) = \chi_f(u,
  v)$.
\end{proof}

From the definition of moved space, it is easy to see that
$\Mov(f_1f_2) \subseteq \Mov(f_1) + \Mov(f_2)$ for any two isometries
$f_1, f_2 \in O(V)$.  \Cref{thm:factorization} allows to construct
factorizations $f = f_1f_2$ where the equality $\Mov(f_1f_2) =
\Mov(f_1) \oplus \Mov(f_2)$ holds.  These are called \emph{direct
  factorizations} in \cite{wall1959structure}.
More generally, we give the following definition.

\begin{definition}[Direct factorization]\label{def:direct-factorization}
  A factorization $f = f_1 \dotsm f_k$ is called a direct
  factorization if $\Mov(f) = \Mov(f_1)\oplus \dotsb
  \oplus \Mov(f_k)$ and no $f_i$ is the identity.
\end{definition}

Recall that the reflections are precisely the isometries with a
one-di\-men\-sio\-nal moved space.  The relation $\Mov(f_1f_2) \subseteq
\Mov(f_1) + \Mov(f_2)$ yields a lower bound on the reflection length
of an isometry $f \in O(V)$: if $f = r_1\dotsm r_k$ is a product of
$k$ reflections, then $\Mov(f) \subseteq \Mov(r_1) +
\dotsb + \Mov(r_k)$, so $k \geq \dim \Mov(f)$.  This lower bound is
attained precisely when the factorization is direct.
In the rest of this section, we are going to see that
most isometries admit a direct factorization, but not all of them.

\begin{lemma}\label{lemma:triangular-basis}
  Let $\chi$ be a non-degenerate bilinear form on a finite-dimensional
  vector space $W$ over a field $\F \neq \F_2$.  If $\chi$ is not
  alternating, then $W$ has a basis $e_1, \dotsc, e_m$ such that
  $\chi(e_i, e_i) \neq 0$ for all $i$, and $\chi(e_i, e_j) = 0$ for $i
  < j$.
\end{lemma}

\begin{proof}
  Since $\chi$ is not alternating, there exists a vector $u \in W$
  such that $\chi(u, u) \neq 0$.  We prove the statement by induction on
  $m = \dim W$, the case $m=1$ being trivial.
  Suppose from now on that $m > 1$.  Note that $W = \< u
  \> \oplus \< u \>^\triangleright$, and that the restriction
  $\chi|_{\< u \>^\triangleright}$ is also non-degenerate.
	
  If $\chi|_{\< u \>^\triangleright}$ is not alternating, we are done
  by induction.  Suppose now by contradiction that $\chi|_{\< u
    \>^\triangleright}$ is alternating.  Let $v \in \< u
  \>^\triangleright$ be a non-zero vector chosen as follows: if $\< u
  \>^\triangleleft \neq \< u \>^\triangleright$, choose $v$ so that
  $\chi(v, u) \neq 0$; otherwise, choose any non-zero vector.  Since
  $\F \neq \F_2$, there exists $a \in \F^\times := \F\setminus \{0\}$ such that $\chi(u +
  av, u + av) = \chi(u, u) + a\chi(v, u)$ does not vanish.  By
  replacing $v$ with $av$, we may assume that $\chi(u+v, u+v) \neq 0$.
  Since $\chi|_{\< u \>^\triangleright}$ is non-degenerate and alternating, the dimension of $\< u
  \>^\triangleright$ is necessarily even, so it is at least $2$.  Then
  there exists a vector $w \in \< u \>^\triangleright$ such that
  $\chi(v, w) = 1$.  Define $c = \chi(u+v, u+v)$, so that for all $b
  \in \F$ we have
  \begin{align*}
    \chi(u+v, u+bv-cw) &= \chi(u, u) + \chi(v, u) - c \chi(v, w) \\
    &= \chi(u, u) + \chi(v, u) - \chi(u+v, u+v) = 0 \\
    \chi(u + bv - cw, u + bv - cw)
    &= \chi(u, u) + b \chi(v, u) - c \chi(w, u).
  \end{align*}
  If $\chi(u, u) + b \chi(v, u) - c \chi(w, u) \neq 0$ for some $b \in
  \F$, then the restriction $\chi|_{\< u + v \>^\triangleright}$ is
  non-degenerate and non-alternating.  Therefore we can set $e_1 =
  u+v$ and be done by induction.  Suppose instead that $\chi(u, u) + b
  \chi(v, u) - c \chi(w, u) = 0$ for all $b \in \F$.  Then $\chi(v, u)
  = 0$ and $\chi(w, u) \neq 0$. In particular, $w \in \< u \>^\triangleright$ and $w \not\in \< u \>^\triangleleft$, so $\< u
  \>^\triangleleft \neq \< u \>^\triangleright$.  This is a
  contradiction because $v$ was chosen so that $\chi(v, u) \neq
  0$.
\end{proof}

\begin{remark}
  It is worth noting that \Cref{lemma:triangular-basis} is false for
  $\F = \F_2$.  See \cite[Chapter 11]{taylor1992geometry} for
  additional details.
\end{remark}

The following lemma describes how the moved space changes when
multiplying an isometry by a reflection.

\begin{lemma}\label{lemma:reflection-multiplication}
  Let $f \in O(V)$ be an isometry, and let $v \in V$ be a non-singular vector.
  \begin{enumerate}[(a)]
  \item If $v \in \Mov(f)$, then $\Mov(r_v f) = \< v
    \>^\triangleright$, where the right orthogonal complement is taken
    inside $\Mov(f)$ with respect to the Wall form $\chi_f$.  In
    particular, $\dim \Mov(r_v f) = \dim \Mov(f) - 1$.
		
  \item If $v \not\in \Mov(f)$, then $\Mov(r_v f) = \Mov(f) \oplus \<
    v \>$.  In particular, $\dim \Mov(r_v f) = \dim \Mov(f) + 1$.
  \end{enumerate}
  As a consequence, if $f$ is a product of $k$ reflections, then $\dim
  \Mov(f) \equiv k \pmod 2$.
\end{lemma}

\begin{proof}
  Part (a) is a direct consequence of \Cref{thm:factorization}.  For
  part (b), suppose that $v \not\in \Mov(f)$.  Since $r_v$ is an
  involution, the fixed space $\Fix(r_v f)$ consists of the vectors $u
  \in V$ such that $r_v(u) = f(u)$.  By substituting the expression
  for $r_v$ (\cref{eq:reflection}), we get the equivalent condition
  \[ u - f(u) = \frac{\beta(u, v)}{Q(v)} v. \]
  Since $v \not\in \Mov(f)$, this condition is satisfied if and only
  if $f(u) = u$ and $\beta(u, v) = 0$.  Therefore $\Fix(r_v f) =
  \Fix(u) \cap \< v\>^\perp$.  Then $\Mov(r_v f) = \Mov(u) \oplus \< v
  \>$ by \Cref{lemma:fix-move-orthogonal}.
\end{proof}

We are now ready to give a simple formula for the reflection length of
any isometry.

\begin{theorem}[Reflection length]\label{thm:reflection-length}
  Assume $\F \neq \F_2$, and let $f \in O(V)$ be an isometry different
  from the identity.  The reflection length of $f$ is equal to
  $\dim\Mov(f)$ if $\Mov(f)$ is not totally singular, and to
  $\dim\Mov(f) + 2$ otherwise.  In particular, every isometry can be
  written as a product of at most $\dim V$ reflections.
\end{theorem}

\begin{proof}
  If $\Mov(f)$ is not totally singular, then
  \Cref{lemma:triangular-basis} applies to the Wall form $\chi_f$ and
  yields a basis of $\Mov(f)$ consisting of non-singular vectors $e_1,
  \dotsc, e_m$ such that $\chi(e_i, e_j) = 0$ for $i < j$.  By a
  repeated application of \Cref{thm:factorization}, we get a direct
  factorization $f = r_{e_1} \dotsm r_{e_m}$ of length $m = \dim
  \Mov(f)$.
	
  Suppose now that $\Mov(f)$ is totally singular.  Choose any
  non-singular vector $v \in V$, and consider $g = r_v f$.  By
  \Cref{lemma:reflection-multiplication}, we have $\Mov(g) = \Mov(f)
  \oplus \< v \>$. In particular, $\Mov(g)$ contains the non-singular
  vector $v$, so by the previous part $g$ can be written as a product
  of $\dim\Mov(g)$ reflections.  Then $f$ can be written as a product
  of $\dim\Mov(g) + 1 = \dim\Mov(f) + 2$ reflections.  It is not
  possible to use less than $\dim\Mov(f) + 2$ reflections: a
  factorization into $\dim\Mov(f)$ reflections would be a direct
  factorization, which does not exist because $\Mov(f)$ is totally
  singular; a factorization into $\dim\Mov(f)+1$ reflections does not
  exist by the last part of \Cref{lemma:reflection-multiplication}.
	
  Finally, we want to show that the reflection length is always at
  most $\dim V$.  This is immediate if $\Mov(f)$ is not totally
  singular, so assume now that $\Mov(f)$ is totally singular.  Since
  $\chi_f$ is non-degenerate, we have $\dim\Mov(f) \geq 2$.  On the
  other hand, $\dim\Mov(f)$ is bounded above by the Witt index of
  $\beta$, which is at most $\frac12 \dim V$.  Therefore the
  reflection length is $\dim\Mov(f) + 2 \leq 2 \, \dim \Mov(f) \leq
  \dim V$.
\end{proof}

In the final part of this section, we introduce the spinor norm following \cite[Section 4]{wall1959structure}.
See also \cite{zassenhaus1962spinor, hahn1979unipotent, scharlau2012quadratic}.
Let $\F^\times = \F \setminus \{ 0 \}$.

\begin{definition}[Wall's spinor norm]
  The \emph{spinor norm} is the map $\theta \colon O(V) \to \F^\times
  / (\F^\times)^2$ defined as $\theta(f) = [\det(A)]$, where $A$ is
  the matrix of $\chi_f$ with respect to any basis of $\Mov(f)$.  Here
  $[a]$ indicates the class of $a \in \F^\times$ in the quotient group
  $\F^\times / (\F^\times)^2$.
\end{definition}

Note that the $\det(A) \neq 0$ because $\chi_f$ is non-degenerate, and
$\theta(f)$ does not depend on the choice of the basis.  For example,
we have $\theta(\id) = 1$ and $\theta(r_v) = [Q(v)]$ for every non-singular vector $v\in V$.

\begin{lemma}\label{lemma:spinor-norm-direct-factorization}
  Given a direct factorization $f = f_1 f_2$, we have $\theta(f) =
  \theta(f_1)\theta(f_2)$.
\end{lemma}

\begin{proof}
  This follows immediately from \Cref{thm:factorization}.
\end{proof}

\begin{theorem}
  The spinor norm is a group homomorphism.
\end{theorem}

\begin{proof}
  If $\F = \F_2$, the spinor norm is trivial, so we can assume from now
  on that $\F\neq \F_2$.  Then $O(V)$ is generated by reflections by
  \Cref{thm:reflection-length}.  Therefore it is enough to show that,
  for every factorization $f = r_1 \dotsm r_k$ into reflections, we
  have $\theta(f) = \theta(r_1) \dotsm \theta(r_k)$.  We prove this by
  induction on $k$, the cases $k=0$ and $k=1$ being trivial.
	
  Fix a length $k$ reflection factorization $f = r_1 \dotsm r_k$ with
  $k \geq 2$. Let $g = r_1 f = r_2 \dotsm r_k$.  If $f = r_1 g$ is a
  direct factorization, then $\theta(f) = \theta(r_1) \theta(g)$ by
  \Cref{lemma:spinor-norm-direct-factorization}.  If $f = r_1 g$ is
  not a direct factorization, then $g = r_1f$ is a direct
  factorization by \Cref{lemma:reflection-multiplication}, and $\theta(g) = \theta(r_1) \theta(f)$ by
  \Cref{lemma:spinor-norm-direct-factorization}.  Since all
  non-trivial elements of $\F^\times / (\F^\times)^2$ have order $2$,
  we have $\theta(f) = \theta(r_1) \theta(g)$ in both cases.  By
  induction, $\theta(g) = \theta(r_2) \dotsm \theta(r_k)$ and
  thus $\theta(f) = \theta(r_1) \theta(g) = \theta(r_1) \dotsm
  \theta(r_k)$.
\end{proof}

\section{Partial order on the orthogonal group}
\label{sec:partial-order}

In this section, we introduce the partial order on $O(V)$ naturally
induced by minimal reflection factorizations. It generalizes the
partial order of \cite{brady2002partial}.  We show that for most
isometries $f \in O(V)$, the interval $[\id,f]$ naturally includes into
the poset (i.e., partially ordered set) of subspaces of $\Mov(f)$.  We
assume throughout this section that $\F \neq \F_2$, so that
\Cref{thm:reflection-length} applies.

\begin{definition}[Partial order on $O(V)$]\label{def:partial-order}
  Given two isometries $f, g \in O(V)$, define $g \leq f$ if and only
  if $f$ admits a minimal length reflection factorization that starts
  with a minimal length reflection factorization of $g$.
  Equivalently, $g \leq f$ if and only if $l(f) = l(g) + l(g^{-1}f)$,
  where $l \colon O(V) \to \N$ denotes the reflection length.
\end{definition}

Since the set of reflections is closed under conjugation, it is
equivalent to require that $f$ admits a minimal factorization that
\emph{ends} with a minimal factorization of $g$.  Notice that $O(V)$
is ranked (in the sense of posets) by the reflection length $l$, and
it has the identity as the unique $\leq$-minimal element.  This
partial order was studied in \cite{brady2002partial} for isometries of
an anisotropic bilinear form $\beta$, and in \cite{brady2015factoring}
for isometries of the affine Euclidean space.

Although the global combinatorics of $O(V)$ is complicated, most of
the intervals \[ [g, f] = \{ h \in O(V) \mid g \leq h \leq f \} \quad
\text{for $g \leq f$} \] have a structure that we can explicitly describe.
Notice that the interval $[g, f]$ is isomorphic (as a poset) to the interval $[\id, g^{-1}f]$ via the isomorphism $h \mapsto g^{-1} h$.  Therefore, the
combinatorial study of all intervals in $O(V)$ reduces to the study of
the intervals of the form $[\id, f]$.

Recall from \Cref{sec:factorizations} that the reflection length of an
isometry $f \in O(V)$ is at least $\dim\Mov(f)$, and the reflection
factorizations of length $\dim\Mov(f)$ (if they exist) are the direct
factorizations.  In light of \Cref{thm:reflection-length}, we can
characterize in a couple of different ways the isometries $f$ with
reflection length equal to $\dim\Mov(f)$.

\begin{definition}\label{def:minimal-isometry}
  An isometry $f \in O(V)$ is \emph{minimal}
  if any of the following equivalent conditions hold:
  \begin{enumerate}[(i)]
  \item $f$ admits a direct factorization as a product of reflections;
  \item its reflection length is equal to $\dim \Mov(f)$;
  \item $f = \id$, or $\Mov(f)$ is not totally singular.
  \end{enumerate}
\end{definition}

Roughly speaking, condition (iii) tells us that most isometries are
minimal.  There are many simple sufficient conditions for an isometry
to be minimal: if $\dim \Mov(f) > \frac12 \dim V$, then $f$ is
minimal; if $\dim\Mov(f)$ is odd, then $f$ is minimal (because all
alternating forms are degenerate, so $\chi_f$ is not alternating); if
$V$ contains no singular vectors, then all isometries are minimal.

\begin{remark}
  If the characteristic of $\F$ is not $2$, there are several
  additional conditions equivalent to \Cref{def:minimal-isometry}.  In
  fact, the moved space $\Mov(f)$ is totally singular if and only if
  $\beta$ vanishes on $\Mov(f)$, which happens if and only if the Wall
  form $\chi_f$ is skew-symmetric (by property (i) of
  \Cref{lemma:wall-form-properties}).  In addition, it is noted in
  \cite[Corollary 6.3]{grove2002classical} that $\Mov(f)$ is totally
  singular if and only if $(f - \id)^2 = 0$ (i.e., the unipotency index
  of $f$ is $2$), or equivalently $\Mov(f) \subseteq \Fix(f)$.
  See also \cite{nokhodkar2017applications}.
\end{remark}

In what follows, we aim to describe the combinatorics of the
interval $[\id,f]$ associated with a minimal isometry $f$.

\begin{lemma}\label{lemma:isometries-below-f}
  Let $f \in O(V)$ be a minimal isometry, and let $g \leq f$. Then:
  \begin{enumerate}[(a)]
  \item $\Mov(g) \subseteq \Mov(f)$;
  \item $g$ is minimal;
  \item $\chi_g$ is the restriction of $\chi_f$ to $\Mov(g)$.
  \end{enumerate}
\end{lemma}

\begin{proof}
  Let $k = \dim \Mov(f)$.  Since $f$ is minimal, its reflection length
  is equal to $k$, and $\Mov(f) = \Mov(r_1) \oplus
  \dotsb \oplus \Mov(r_k)$ for every minimal length factorization
  $f=r_1\dotsm r_k$ of $f$ as a product of reflections.  Then there is
  one such factorization for which $g = r_1 \dotsm r_m$ for some $m
  \leq k$, and the reflection length of $g$ is equal to $m$.  By a
  repeated application of part (b) of
  \Cref{lemma:reflection-multiplication}, we get that $\Mov(g) =
  \Mov(r_1) \oplus \dotsb \oplus \Mov(r_m) \subseteq
  \Mov(f)$.  In addition, the reflection factorization $g = r_1\dotsm
  r_m$ is a direct factorization, so $g$ is minimal.
    
  If $g = f$, then $\chi_g = \chi_f$ and we are done.  Suppose now that
  $g \neq f$, i.e., $m < k$.  Since $\Mov(r_k)$ is 1-dimensional, the
  property $\chi(u,u) = Q(u)$ (\Cref{thm:wall-form}) implies that
  $\chi_{r_k}$ is the restriction of $\chi_f$ to $\Mov(r_k)$.  By
  \Cref{thm:factorization}, $\chi_{r_1\dotsm r_{k-1}}$ is the
  restriction of $\chi_f$ to $\Mov(r_1\dotsm r_{k-1})$.  Now, $f' :=
  r_1\dotsm r_{k-1}$ is minimal by part (b), and $g \leq f'$, so we
  are done by induction on $k$.
\end{proof}

In the full group $O(V)$, there can be many isometries with the same
moved space.  However, once we restrict to an interval $[\id, f]$ where $f$ is minimal, an isometry is completely determined by its moved space.

\begin{theorem}[Minimal intervals]\label{thm:intervals}
  Let $f \in O(V)$ be a minimal isometry.  Then $g \mapsto \Mov(g)$ is
  an order-preserving bijection between the interval $[\id,f]$ and the
  poset of linear subspaces $U \subseteq \Mov(f)$ that satisfy the
  following conditions:
  \begin{enumerate}[(i)]
  \item $U = \{0\}$ or $U$ is not totally singular;
  \item $U^\triangleright = \{0\}$ or $U^\triangleright$ is not
    totally singular;
  \item $\chi_f|_U$ is non-degenerate.
  \end{enumerate}
  In addition, the rank of $g \in [\id,f]$ is equal to $\dim\Mov(g)$.
\end{theorem}

\begin{proof}
  Let $g \in [\id,f]$, and let $U = \Mov(g)$. We have that $g$ is
  minimal by \Cref{lemma:isometries-below-f}, so $U$ satisfies
  condition (i).  In addition, we have $U^\triangleright = \Mov(g^{-1}f)$ by
  \Cref{thm:factorization}, and $g^{-1}f \in [\id,f]$ is also minimal,
  so condition (ii) is satisfied.  Finally, condition (iii) is a
  consequence of \Cref{thm:factorization}.
	
  We now explicitly construct the inverse map $\phi$.  Suppose that $U
  \subseteq \Mov(f)$ satisfies all three conditions.  By
  \Cref{thm:factorization} and condition (iii), there is a direct
  factorization $f = f_1f_2$ where $f_1$ is the isometry associated
  with $(U, \chi_f|_U)$.  By conditions (i) and (ii), both $f_1$ and
  $f_2$ are minimal.  Then their reflection lengths are
  $\dim\Mov(f_1)$ and $\dim\Mov(f_2)$, which add up to $\dim\Mov(f)$.
  Therefore $f_1 \in [\id,f]$.  Define $\phi(U) = f_1$.
	
  We now check that $\phi$ is indeed the inverse of $\Mov$.  For any
  isometry $g \in [\id,f]$, we have that $g' = \phi(\Mov(g))$ is an
  isometry such that $\Mov(g') = \Mov(g)$, and $\chi_{g'} =
  \chi_f|_{\Mov(g)}$.  By \Cref{lemma:isometries-below-f}, we also
  have that $\chi_g = \chi_f|_{\Mov(g)}$.  This means that $g'$ and
  $g$ have the same moved space and the same Wall form, so $g' = g$ by
  \Cref{thm:wall-parametrization}.  In addition, for any subspace $U
  \subseteq \Mov(f)$ satisfying conditions (i)-(iii), we have that
  $\Mov(\phi(U)) = U$ by construction of $\phi$.
	
  If $g \leq g'$ in $[\id,f]$, then $g'$ is minimal by part (b) of
  \Cref{lemma:isometries-below-f}, and $\Mov(g) \subseteq \Mov(g')$ by
  part (a) of \Cref{lemma:isometries-below-f}.  This means that the
  bijection $g \mapsto \Mov(g)$ is order-preserving.  Finally, the
  rank of an isometry $g$ in $[\id,f]$ is given by its reflection
  length, which is equal to $\dim\Mov(g)$ because $g$ is minimal.
\end{proof}

For every $U \subseteq \Mov(f)$, we have that $U^\triangleleft =
f(U^\triangleright)$ by property (ii) of
\Cref{lemma:wall-form-properties}, so $U^\triangleleft$ and
$U^\triangleright$ are isometric.  In particular, $U^\triangleleft$ is
totally singular if and only if $U^\triangleright$ is totally
singular, and this gives an equivalent way to write condition (ii) of
\Cref{thm:intervals}.  Note that condition (ii) is not redundant, due
to the following example.

\begin{example}\label{example:direct-factorization}
  Consider an isometry $f$ with a $3$-dimensional moved space and a
  Wall form given by the following matrix, with respect to some basis
  $e_1, e_2, e_3$ of $\Mov(f)$:
  \[ \begin{mymatrix}{ccc}
    1 & 0 & 0 \\
    0 & 0 & 1 \\
    0 & -1 & 0
  \end{mymatrix}. \]
  If $U_1 = \< e_1 \>$ and $U_2 = U_1^\triangleright = \< e_2, e_3
  \>$, then \Cref{thm:factorization} yields a direct factorization $f
  = f_1 f_2$ such that $\chi_{f_1} = \chi_f|_{U_1}$ is not
  alternating, whereas $\chi_{f_2} = \chi_f|_{U_2}$ is alternating.  Then $f_1$ is
  minimal, and $f_2$ is not.  As a consequence, we have $f_1 \not\leq
  f$ despite the inclusion $\Mov(f_1) \subseteq \Mov(f)$.
\end{example}

Notice that the bijection $g \mapsto \Mov(g)$ of \Cref{thm:intervals}
is not a poset isomorphism.  Indeed, it is possible to have elements
$g, g' \in [\id,f]$ with $g \not\leq g'$ but $\Mov(g) \subseteq
\Mov(g')$.  We construct such a case in the following example.

\begin{example}
  Consider an isometry $f$ with a $4$-dimensional moved space and a
  Wall form given by the following matrix, with respect to some basis
  $e_1, e_2, e_3, e_4$ of $\Mov(f)$:
  \[ \begin{mymatrix}{cccc}
    1 & 0 & 0 & 0 \\
    0 & 0 & 1 & 0 \\
    0 & -1 & 0 & 0 \\
    0 & 0 & 0 & 1
  \end{mymatrix}. \]
  By \Cref{thm:intervals}, the subspaces $U = \< e_1 \>$ and $U' = \<
  e_1, e_2, e_3 \>$ have associated isometries $g, g' \in [\id,f]$
  with $\Mov(g) = U$ and $\Mov(g') = U'$.  Then $\Mov(g) \subseteq
  \Mov(g')$, but $g \not\leq g'$ as seen in
  \Cref{example:direct-factorization}.
\end{example}

In the case where the bilinear form $\beta$ is anisotropic, we recover
the description of the intervals in $O(V)$ given in
\cite{brady2002partial}.
In fact, the same description is obtained in the more general setting where $V$ contains no singular vectors.

\begin{corollary}\label{cor:anisotropic-intervals}
  Suppose that $V$ contains no singular vectors, and let $f \in O(V)$ be any isometry.
  Then $f$ is minimal, and $g \mapsto \Mov(g)$ is
  an isomorphism between the interval $[\id,f]$ and the
  poset of all linear subspaces $U \subseteq \Mov(f)$.
\end{corollary}

\begin{proof}
  We already noted that every isometry $f$ is minimal if $V$ contains no singular vectors.
  To prove that $g \mapsto \Mov(g)$ is an order-preserving bijection, it is enough to apply \Cref{thm:intervals} and show that conditions (i)--(iii) are satisfied by every subspace $U \subseteq \Mov(f)$.
  Conditions (i) and (ii) are trivially satisfied because $\{0\}$ is the only totally singular subspace of $V$.
  For condition (iii), $\chi_f(u, u) = Q(u) \neq 0$ for any non-zero vector $u \in U$, so $\chi_f|_U$ is non-degenerate.
  To conclude the proof, we need to show that $\Mov(g) \subseteq \Mov(g')$ implies $g \leq g'$ for every $g,g' \in [\id, f]$.
  If we define $h = g^{-1}g'$, we obtain that $g' = gh$ is a direct factorization by \Cref{thm:factorization}.
  Since $h$ is minimal, we deduce that $l(g') = l(g) + l(h)$ and therefore $g \leq g'$.
\end{proof}

In the last part of this section, we turn our attention to non-minimal isometries, which behave in a substantially different way.

\begin{theorem}
  \label{thm:non-minimal-isometries}
  Let $f \in O(V)$ be a non-minimal isometry.
  \begin{enumerate}[(a)]
  \item For every reflection $r \in O(V)$, we have $r < f$ and $rf < f$.
  \item Every isometry $g < f$ is minimal.
  \item $f$ is $\leq$-maximal in $O(V)$.
  \end{enumerate}
\end{theorem}

\begin{proof}
  In the proof of \Cref{thm:reflection-length}, it is shown that any
  reflection $r \in O(V)$ is part of some minimal length reflection
  factorization of $f$.  This implies both $r \leq f$ and $rf \leq f$.
  Note that $r \neq f$ because every reflection is minimal, and
  clearly $rf \neq f$, so the strict relations of part (a) hold.  From
  that proof it is also clear that $rf$ is minimal, so every isometry
  $g < f$ is minimal by \Cref{lemma:isometries-below-f}, proving part
  (b).  Part (c) follows from \Cref{lemma:isometries-below-f} and part
  (b).
\end{proof}

In the following, we give a coarse description of the structure of $[\id, f]$ for a non-minimal isometry $f$.
Note that $[\id, f]$ contains multiple isometries with the same moved space, so a bijection like the one of \Cref{thm:intervals} does not exist.
Denote by $(\id, f) = [\id, f] \setminus \{\id, f\}$ the open interval between the identity and $f$.
Let $\W_f$ be the set of all
subspaces $W \subseteq V$ containing $\Mov(f)$ as a codimension-one
subspace and not totally singular.
For any subspace $W \in
\W_f$, let $P_{f, W} = \{ g \in (\id, f) \mid \Mov(g) \subseteq W \}$.

\begin{theorem}[Non-minimal intervals]
	Let $f \in O(V)$ be a non-minimal isometry. As a poset, the open interval $(\id, f)$ is the disjoint union (also called ``parallel composition'') of the subposets $P_{f, W}$:
	\[ (\id, f) = \bigsqcup_{W \in \W_f} P_{f, W}. \]
	\label{thm:non-minimal-intervals}
\end{theorem}

\begin{proof}
	Let $g \in (\id, f)$. Then $g \leq rf$ for some reflection $r$, and $rf$ is minimal by \Cref{thm:non-minimal-isometries}.
	Since $f$ is non-minimal, $\Mov(f)$ is a codimension-one subspace of $W = \Mov(rf)$ by part (b) of \Cref{lemma:reflection-multiplication}.
	Then $W \in \W_f$ because $rf$ is minimal, and $g \in P_{f, W}$ by \Cref{lemma:isometries-below-f}.
	
	Let $W' \in \W_f$ be any subspace such that $g \in P_{f, W'}$.
	Note that $g$ is minimal by \Cref{thm:non-minimal-isometries}, so $\Mov(g) \nsubseteq \Mov(f)$.
	Since $\Mov(f)$ is a
	codimension-one subspace of $W'$, we have that $W' =
	\Mov(f) + \Mov(g)$.
	Therefore $W'$ is uniquely determined by $f$ and $g$.
	In other words, $g$ is contained in exactly one $P_{f, W'}$.
	
	Finally, if $g \in P_{f, W}$ and $g' \leq g$, then $\Mov(g') \subseteq \Mov(g)$ by \Cref{lemma:isometries-below-f} and therefore $g' \in P_{f, W} \cup \{\id\}$.
	This means that there is no order relation between $P_{f, W}$ and $P_{f, W'}$ if $W \neq W'$.
\end{proof}

\pgfdeclarelayer{bg}    %
\pgfsetlayers{bg,main}  %

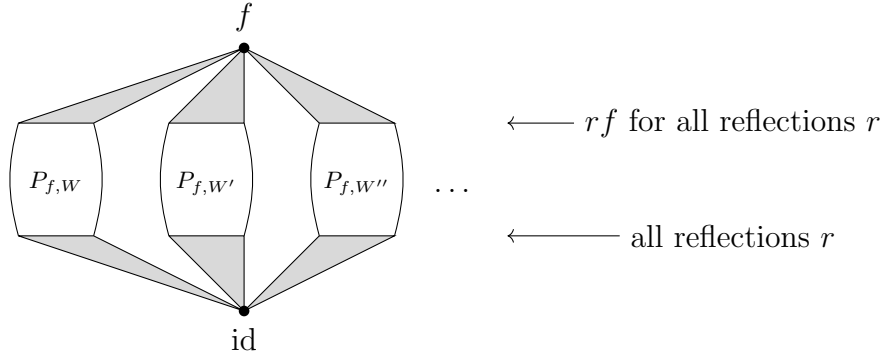
\begin{figure}[t]
	\begin{tikzpicture}
		\newcommand{\atoms}{1}
		\newcommand{\coatoms}{2.5}
		\node[inner sep=0.05cm, fill, circle, label=below:{$\id$}] (id) at (0,0) {};
		\node[inner sep=0.05cm, fill, circle, label=above:{$f$}] (f) at (0,3.5) {};

		\draw (-3, \atoms) to [bend left=15] (-3, \coatoms);
		\draw (-2, \atoms) to [bend right=15] (-2, \coatoms);
		\node at (-2.5, 1.7) {\scriptsize $P_{f,W}$};
		
		\draw (-1, \atoms) to [bend left=15] (-1, \coatoms);
		\draw (0, \atoms) to [bend right=15] (0, \coatoms);
		\node at (-0.5, 1.7) {\scriptsize $P_{f,W'}$};
		
		\node at (2.8, 1.6) {$\dots$};

		\draw (2, \atoms) to [bend right=15] (2, \coatoms);
		\draw (1, \atoms) to [bend left=15] (1, \coatoms);
		\node at (1.5, 1.7) {\scriptsize $P_{f,W''}$};
		
		\begin{pgfonlayer}{bg}    %
			\draw[fill=gray!30] (-3, \atoms) -- (-2, \atoms) -- (id.center) -- (-3, \atoms);
			\draw[fill=gray!30] (-1, \atoms) -- (0, \atoms) -- (id.center) -- (-1, \atoms);
			\draw[fill=gray!30] (1, \atoms) -- (2, \atoms) -- (id.center) -- (1, \atoms);
					
			\draw[fill=gray!30] (-3, \coatoms) -- (-2, \coatoms) -- (f.center) -- (-3, \coatoms);
			\draw[fill=gray!30] (-1, \coatoms) -- (0, \coatoms) -- (f.center) -- (-1, \coatoms);
			\draw[fill=gray!30] (1, \coatoms) -- (2, \coatoms) -- (f.center) -- (1, \coatoms);
		\end{pgfonlayer}
		
		\node (rf) at (6.5, \coatoms) {$rf$ for all reflections $r$};
		\draw[->] (rf) -- (3.5, \coatoms);
		\node (r) at (6.5, \atoms) {all reflections $r$};
		\draw[->] (r) -- (3.5, \atoms);
	\end{tikzpicture}
	
	\caption{Coarse structure of an interval $[\id, f]$ for a non-minimal isometry $f$, as described by \Cref{thm:non-minimal-intervals}.}
	\label{fig:non-minimal-interval}
\end{figure}

\Cref{fig:non-minimal-interval} shows the Hasse diagram of a non-minimal interval $[1,f]$, as described by the previous theorem.
Note that each subposet $P_{f, W}$ is self-dual: the map $g \mapsto g^{-1}f$ is an order-reversing bijection from $P_{f, W}$ to itself.

\section{Positive factorizations}
\label{sec:positive-factorizations}

Let $(V, Q)$ be a non-degenerate quadratic space over an ordered field
$\F$.  In particular, $\F$ has characteristic $0$.  A non-singular
vector $v \in V$ is said to be \emph{positive} if $Q(v) > 0$, and
\emph{negative} if $Q(v) < 0$.  In this section we focus on the
factorizations of isometries into \emph{positive reflections},
i.e., reflections with respect to positive vectors.  We refer to these
factorizations as \emph{positive reflection factorizations}.  Under
the hypothesis that $\F$ is \emph{square-dense} (the squares are dense
in the positive elements), we obtain a clean description of the
minimal length of a positive reflection factorization of any isometry
$f \in O(V)$.  In particular, we show that $f$ admits a positive
reflection factorization if and only if its spinor norm is positive.

Recall that a subspace $W \subseteq V$ is \emph{positive definite}
(resp.\ \emph{negative definite}) if $Q(v) > 0$ (resp. $< 0$) for
every non-zero vector $v \in W$.  It is \emph{positive semi-definite}
(resp.\ \emph{negative semi-definite}) if $Q(v) \geq 0$ (resp.\ $\leq
0$) for all $v \in W$.  By the inertia theorem of Jacobi and Sylvester
\cite[Theorem 4.4]{scharlau2012quadratic}, $V$ can be decomposed as an
orthogonal direct sum $V^+ \perp V^-$, where $V^+$ is a positive
definite subspace and $V^-$ is a negative definite subspace.  The
dimensions of $V^+$ and $V^-$ do not depend on the chosen
decomposition, and the pair $(\dim V^+, \dim V^-)$ is called the
\emph{signature} of $(V, Q)$.
We refer to \cite{scharlau2012quadratic} for additional theory on
quadratic spaces over ordered fields.  We assume from now on that $V$
is not negative definite, because otherwise there are no positive
vectors.

Denote by $\F^+ \subseteq \F$ the subset of all positive elements of $\F$.
Since $(\F^\times)^2 \subseteq \F^+$, there is a well-defined quotient
map $\pi \colon \F^\times / (\F^\times)^2 \to \F^\times / \ \F^+ \cong
\Z_2$. In other words, every element of $\F^\times / (\F^\times)^2$ is
either positive or negative, and this notion is well-defined.

\begin{definition}
  An isometry $f \in O(V)$ is \emph{positive} (resp.\ negative) if its
  spinor norm $\theta(f)$ is positive (resp.\ negative).
\end{definition}

Notice that this definition is compatible with the previous definition
of positive reflection: a reflection $r_v$ is positive if and only if
$Q(v) > 0$.  The positive isometries form a subgroup $O_+(V)$ of
$O(V)$, being the kernel of the composition \[ O(V) \xrightarrow{\theta}
\F^\times / (\F^\times)^2 \xrightarrow{\pi} \Z_2. \] In particular, if
an isometry $f \in O(V)$ can be written as a product of positive
reflections, then it is positive.  The subgroup $O_+(V)$ has index $2$
in $O(V)$ unless $V$ is positive definite, in which case $O_+(V) =
O(V)$.

\begin{example}[Isometries over the real numbers]
  If $\F = \R$ and $V$ is not (positive or negative) definite, then
  $O(V)$ has four connected components.  They are detected by the
  surjective group homomorphism $O(V) \to \Z_2 \times \Z_2$ defined as
  $f \mapsto (\pi(\theta(f)), \det(f))$.  The connected component of the identity is
  $O_+(V) \cap SO(V)$.
\end{example}

We are interested in determining the \emph{positive reflection length}
of a positive isometry $f \in O_+(V)$, i.e., the minimal length of a
positive reflection factorization of $f$.  A lower bound for
the positive reflection length is given by the reflection length, which is computed in \Cref{thm:reflection-length}.  The following example shows
that this lower bound is not always attained.

\begin{example}
  Suppose that $W \subseteq V$ is a $2$-dimensional negative definite
  subspace, and let $\chi = \frac12 \beta|_{W}$.  Let $f \in O(V)$ be
  the isometry with $\Mov(f) = W$ and $\chi_f = \chi$.  Then $f$ is
  positive and minimal (in the sense of \Cref{def:minimal-isometry}),
  but all the reflections $r \leq f$ are negative.  Therefore $f$ is a
  product of $2$ negative reflections, but it cannot be written as a
  product of $2$ positive reflections.  Note that $f$ is an
  involution, by property (v) of \Cref{lemma:wall-form-properties}.
\end{example}

More generally, if $f$ is an involution, we have $\chi_f = \frac12
\beta|_{\Mov(f)}$ by properties (i) and (v) of
\Cref{lemma:wall-form-properties}.  Then a triangular basis (as in
\Cref{lemma:triangular-basis}) of positive vectors exists if and only
if $\Mov(f)$ is positive definite.  In other words, an involution $f$
admits a direct factorization into positive reflections if and only if
$\Mov(f)$ is positive definite.

We aim to show that all positive non-involutions admit a direct
factorization into positive reflections provided that $\Mov(f)$
contains at least one positive vector.  To prove this, in the
rest of this section, we are going to assume that the field $\F$
satisfies the following property.

\begin{definition}
  An ordered field $\F$ is \emph{square-dense} if the set of squares
  $(\F^\times)^2$ is dense in the set of positive elements $\F^+$.  In
  other words, for every $0 < a < b$, there exists a square $c^2$ such
  that $a < c^2 < b$.
\end{definition}

The class of square-dense fields includes all \emph{Archimedean
  fields} (i.e., the subfields of $\R$) and \emph{Euclidean fields}
(i.e., ordered fields where every positive element is a square), which
include all \emph{real closed fields}.  See \cite[Chapter
  3]{scharlau2012quadratic} for the definitions and properties of
these classes of fields, particularly in relation to the theory of
quadratic forms.  An example of an ordered field that is not
square-dense is the field of rational functions $\Q(X)$, with the
order determined by $a < X$ for all $a \in \Q$ (this is a typical
example of a non-Archimedean field).

Our reason to choose the square-dense property as our working
hypothesis is that it is quite general, but at the same time, it allows us
to obtain the same characterization of the positive reflection length
(\Cref{thm:positive-factorizations}) that we would obtain over the
real numbers.

We start by proving a variant of \Cref{lemma:triangular-basis}.

\begin{lemma}\label{lemma:basis-with-one-positive-vector}
  Let $\chi$ be a non-degenerate bilinear form on a finite-dimensional
  vector space $W$ over an ordered field $\F$, with $\dim W \geq 2$.
  Suppose that there is at least one vector $u \in W$ with $\chi(u, u)
  > 0$.  Then there is a basis $e_1, \dotsc, e_m$ such that $\chi(e_1,
  e_1) > 0$, $\chi(e_i, e_i) \neq 0$ for $i \geq 2$, and $\chi(e_i,
  e_j) = 0$ for $i<j$.
\end{lemma}

\begin{proof}
  Proceed as in the proof of \Cref{lemma:triangular-basis}, starting
  with a vector $u$ such that $\chi(u,u) > 0$.  Choose $a \in
  \F^\times$ such that $\chi(u, u) + a \chi(v, u) > 0$, for example by
  taking $a = \chi(v, u)$.  Then the first basis vector $e_1$
  satisfies $\chi(e_1, e_1) > 0$.  The rest of the proof is unchanged.
\end{proof}

Next, we prove a technical lemma in dimension 3.  This is the building
block that allows us to construct triangular bases of positive
vectors when the Wall form is not symmetric.

\begin{lemma}\label{lemma:non-symmetric-3d}
  Let $W$ be a $3$-dimensional vector space over a square-dense field
  $\F$.  Let $\chi$ be a non-degenerate bilinear form on $W$.  Suppose
  that $\chi$ is not symmetric, and that there is at least one vector
  $u \in W$ with $\chi(u,u) > 0$.  Then there exist two vectors $v_1,
  v_2 \in W$ such that $\chi(v_1, v_1) > 0$, $\chi(v_2, v_2) > 0$, and
  $\chi(v_1, v_2) = 0$.
\end{lemma}

\begin{proof}
  By \Cref{lemma:basis-with-one-positive-vector}, there exists a
  vector $e_1 \in W$ such that $\chi(e_1, e_1) > 0$ and $\chi|_{\< e_1
    \>^\triangleright}$ is not alternating.  Fix any non-zero vector
  $e_2 \in \< e_1 \>^\triangleleft \cap \< e_1 \>^\triangleright$.  If
  $\chi(e_2, e_2) > 0$, we are done by choosing $v_1 = e_1$ and $v_2 =
  e_2$.  So we may assume that $\chi(e_2, e_2) \leq 0$.
	
  \emph{Case 1: $\chi(e_2, e_2) = 0$.}  Since $\chi|_{\< e_1
    \>^\triangleright}$ is not alternating, there exists a vector $e_3
  \in \< e_1 \>^\triangleright$ such that $\chi(e_3, e_3) \neq 0$.  If
  $\chi(e_3, e_3) > 0$, we are done by choosing $v_1 = e_1$ and $v_2 =
  e_3$.  So we can assume that $\chi(e_3, e_3) < 0$.  Note that $e_3$
  is not a scalar multiple of $e_2$, so $e_2, e_3$ is a basis of $\<
  e_1 \>^\triangleright$.  Therefore $e_1, e_2, e_3$ is a basis of
  $W$, and in this basis the matrix of $\chi$ has the following form:
  \[
  \begin{mymatrix}{ccc}
    \gamma & 0 & 0 \\
    0 & 0 & c \\
    a & b & -\delta
  \end{mymatrix},
  \]
  with $\gamma, \delta > 0$, and $b, c \neq 0$ (otherwise $\chi$ is
  degenerate).
  We may also assume $a \neq 0$ since otherwise we can exchange $e_2$ and $e_3$ and reduce to the case 2 below.
	
  If $b + c \neq 0$, then set $v_1 = e_1$ and $v_2 = 2\delta e_2 +
  (b+c) e_3$.  We have that $\chi(v_1, v_2) = 0$, and $\chi(v_2, v_2)
  = \delta (b+c)^2 > 0$, so we are done.  Suppose now that $b + c =
  0$, so the matrix of $\chi$ becomes
  \[
  \begin{mymatrix}{ccc}
    \gamma & 0 & 0 \\
    0 & 0 & -b \\
    a & b & -\delta
  \end{mymatrix}.
  \]
  Let $v_1 = ab e_1 + \gamma\delta e_2$ and $v_2 = \delta e_1 + a
  e_3$.  Then
  \begin{align*}
    \chi(v_1, v_1) &= \gamma (ab)^2 > 0 \\
    \chi(v_1, v_2) &= \gamma \cdot ab \cdot \delta - b \cdot  \gamma\delta \cdot a = 0 \\
    \chi(v_2, v_2) &= \gamma \delta^2 + a \cdot \delta \cdot a - \delta a^2 = \gamma\delta^2 > 0.
  \end{align*}
  
  \emph{Case 2: $\chi(e_2, e_2) < 0$.}  Then $\chi|_{\<e_1, e_2\>}$ is
  non-degenerate, and $\< e_1, e_2 \> \cap \< e_1, e_2
  \>^\triangleright = \{0\}$.  Let $e_3 \in \< e_1, e_2
  \>^\triangleright$ be any non-zero vector.  Note that $\chi(e_3,
  e_3) \neq 0$, because $\chi$ is non-degenerate.  If $\chi(e_3, e_3)
  > 0$, we are done by setting $v_1 = e_1$ and $v_2 = e_3$, so we can
  assume that $\chi(e_3, e_3) < 0$.  Then the matrix of $\chi$ with
  respect to the basis $e_1, e_2, e_3$ has the following form:
  \[
  \begin{mymatrix}{ccc}
    \gamma & 0 & 0 \\
    0 & -\delta & 0 \\
    a & b & -\epsilon
  \end{mymatrix},
  \]
  where $\gamma, \delta, \epsilon > 0$, and at least one of $a$ and
  $b$ is non-zero (because $\chi$ is not symmetric).  Define
  \begin{align*}
    v_1 &= qe_1 + e_2 \\ v_2 &= e_1 + \frac{\gamma}{\delta} q e_2 +
    \frac{1}{2\epsilon} \left(a + \frac{\gamma}{\delta}bq \right) e_3,
  \end{align*}
  where $q \in \F$ is yet to be determined.  Then
  \begin{align*}
    \chi(v_1, v_1) &= \gamma q^2 - \delta \\ \chi(v_1, v_2) &= \gamma
    q - \delta \cdot \frac{\gamma}{\delta}q = 0 \\ \chi(v_2, v_2) &=
    \gamma - \frac{\gamma^2}{\delta} q^2 + \frac{1}{4\epsilon} \left(a
    + \frac{\gamma}{\delta}bq \right)^2.
  \end{align*}
  We are going to show how to choose $q$ so that $\chi(v_1, v_1) > 0$
  and $\chi(v_2, v_2) > 0$.  The first condition is
  \begin{equation}
    q^2 > \frac{\delta}{\gamma}.
    \label{eq:first-inequality}
  \end{equation}
  Now fix the sign of $q$ so that $abq \geq 0$.
  Then
  \[ \chi(v_2, v_2) \geq
  \gamma - \frac{\gamma^2}{\delta}q^2 + \frac{1}{4\epsilon} \left( a^2 +
  \left(\frac{\gamma}{\delta}b \right)^2 q^2 \right). \] In order to have
  $\chi(v_2, v_2) > 0$, it is enough to have that the right hand side
  of the previous equation is positive, and this condition can be
  rewritten as
  \begin{equation}\label{eq:second-inequality}
    \left(1 - \frac{b^2}{4\delta\epsilon} \right) q^2 < \left( 1 +
    \frac{a^2}{4\gamma\epsilon} \right) \frac{\delta}{\gamma}.
  \end{equation}
  If $b^2 \geq 4\delta\epsilon$, then \cref{eq:second-inequality} is always
  satisfied, and \cref{eq:first-inequality} is satisfied for
  \[ q = \pm \left(\frac{\delta}{\gamma} + 1 \right). \]
  If $b^2 < 4\delta\epsilon$, then
  \cref{eq:first-inequality,eq:second-inequality} are satisfied if
  \[ \frac{\delta}{\gamma} < q^2 < \frac{1 + a^2/4\gamma\epsilon}{1 -
    b^2/4\delta\epsilon} \cdot \frac{\delta}{\gamma} \] Recall that at least
  one of $a$ and $b$ is non-zero, so these inequalities define a
  non-empty interval in $\F^+$.  Since $\F$ is square-dense, this
  interval contains at least one square $q^2$.
\end{proof}

It is worth mentioning that \Cref{lemma:non-symmetric-3d} does not
hold over a general ordered field $\F$, as we show in the next
example.

\begin{example}\label{example:counterexample-positive-factorization}
  Let $\F = \Q(X)$, with the non-Archimedean order determined by $a <
  X$ for all $a \in \Q$.  On $W = \F^3$, consider the non-symmetric
  bilinear form $\chi$ defined by the following matrix:
  \[
  \begin{mymatrix}{ccc}
    1 & 0 & 0 \\
    0 & -X & 0 \\
    0 & 1 & -X
  \end{mymatrix}.
  \]
  Let $v = (p, q, r) \in W$ be any vector satisfying $\chi(v, v) > 0$.
  Then we have $p^2 - Xq^2 - Xr^2 + qr > 0$.  Note that $\deg(qr)
  < \max \{ \deg(Xq^2), \deg(Xr^2) \}$, unless both $q$ and $r$ are
  zero.  Therefore we must have $\deg(p^2) \geq \max \{ \deg(Xq^2),
  \deg(Xr^2) \}$, which can be rewritten as $\deg(p) > \deg(q)$ and
  $\deg(p) > \deg(r)$.  Now, suppose to have two vectors $v_1 = (p_1,
  q_1, r_1)$, $v_2 = (p_2, q_2, r_2)$ with $\chi(v_1, v_1) > 0$ and
  $\chi(v_2, v_2) > 0$. Then $\chi(v_1, v_2) = p_1p_2 - Xq_1q_2 -
  Xr_1r_2 + r_1q_2$, and here the degree of $p_1p_2$ is greater than
  the degree of all other terms.  Therefore $\chi(v_1, v_2) \neq 0$.
\end{example}

We are going to need some flexibility in the choice of the vectors
$v_1, v_2$ given by \Cref{lemma:non-symmetric-3d}.  The following two
easy lemmas allow us to modify a pair $(v_1, v_2)$ while maintaining the
properties we need.

\begin{lemma}\label{lemma:perturb-orthogonal-pair}
  Let $W$ be a finite-dimensional vector space over an ordered field
  $\F$, with $\dim W \geq 2$.  Let $\chi$ be a non-degenerate bilinear
  form on $W$, and suppose to have two non-zero vectors $v_1, v_2 \in
  W$ with $\chi(v_1, v_2) = 0$.  For every $u \in W$, there exists a
  vector $w \in W$ such that $\chi(v_1 + a u, v_2 + a w)
  = 0$ for all $a \in \F$.
\end{lemma}

\begin{proof}
  If $u \in \< v_1 \>$, then we can simply choose $w = 0$.  Suppose
  now that $u \not \in \<v_1\>$.  Then $\< v_1 \>^\triangleright$ and
  $\< u \>^\triangleright$ are two distinct hyperplanes of $W$.  The
  set $H = \{ w \in W \mid \chi(u, v_2) + \chi(v_1, w) = 0 \}$ is an
  affine translate of $\< v_1 \>^\triangleright$, and so it intersects
  the linear hyperplane $\< u \>^\triangleright$. Let $w \in H \cap \<
  u \>^\triangleright$.  Then
  \[ \chi(v_1 + a u, v_2 + a
  w) = \chi(v_1, v_2) + a \big(\chi(u, v_2) + \chi(v_1, w)
  \big) + a^2 \chi(u, w) = 0 \]
  for all $a \in \F$.
\end{proof}

\begin{lemma}\label{lemma:perturb-positive-vector}
  Let $W$ be a finite-dimensional vector space over an ordered field
  $\F$.  Let $\chi$ be a non-degenerate bilinear form on $W$, and
  suppose to have a vector $v \in W$ with $\chi(v, v) > 0$.  For every
  $u \in W$, there exists $\delta \in \F^+$ such that $\chi(v+
  a u, v + a u) > 0$ for all $a$ in the open
  interval $(-\delta, \delta)$.
\end{lemma}

\begin{proof}
  We have
  \[ \chi(v+ a u, v + a u) = \chi(v, v) + a
  \chi(u, v) + a \chi(v, u) + a^2 \chi(u, u). \]
  The absolute value of the last three summands can be made smaller than
  $\frac13 \chi(v, v)$, for a sufficiently small $a$.
\end{proof}

We are finally able to refine
\Cref{lemma:basis-with-one-positive-vector}, and obtain a whole
triangular basis of positive vectors.

\begin{lemma}\label{lemma:positive-basis}
  Let $W$ be a finite-dimensional vector space over a square-dense
  field $\F$.  Let $\chi$ be a non-degenerate bilinear form on $W$
  with $\det(\chi) > 0$.  Suppose that $\chi$ is not symmetric, and
  that there is at least one vector $u \in W$ with $\chi(u,u) > 0$.
  Then $W$ has a basis $e_1, \dotsc, e_m$ such that $\chi(e_i, e_i) >
  0$ for all $i$, and $\chi(e_i, e_j) = 0$ for $i<j$.
\end{lemma}

\begin{proof}
  The proof is by induction on $m = \dim W$, the case $m=1$ being
  trivial.  By \Cref{lemma:basis-with-one-positive-vector}, there is a
  basis $e_1, \dotsc, e_m$ such that $\chi(e_1, e_1) > 0$, $\chi(e_i,
  e_i) \neq 0$ for $i \geq 2$, and $\chi(e_i, e_j) = 0$ for $i<j$.  If
  $m=2$, since $\det(\chi) > 0$, we deduce that $\chi(e_2, e_2) > 0$
  and we are done.  Assume from now on that $m \geq 3$.
	
  Since $\chi$ is not symmetric, there exist two indices $2 \leq i < j
  \leq m$ such that at least one of $\chi(e_i, e_1)$, $\chi(e_j,
  e_1)$, $\chi(e_j, e_i)$ is not zero.  Apply
  \Cref{lemma:non-symmetric-3d} to the restriction of $\chi$ to the
  3-dimensional subspace $U = \< e_1, e_i, e_j\>$ and get two
  positive vectors $v_1, v_2 \in U$ such that $\chi(v_1, v_2) = 0$.
  In particular, the subspace $\< v_1 \>^\triangleright$ contains the
  positive vector $v_2$ (here the right orthogonal complement is taken
  in the entire space $W$ with respect to the bilinear form $\chi$).
	
  By
  \Cref{lemma:perturb-orthogonal-pair,lemma:perturb-positive-vector},
  there exists $a \in \F^\times$ such that for all $i=1, \dotsc,
  m$ we have: (1) $\chi(v_1 + a e_i, v_1 + a e_i) > 0$; (2) the
  subspace $\< v_1 + a e_i \>^\triangleright$ contains some
  positive vector $v_2 + a e_i'$.  Let $N = \{ v_1, v_1 +
  a e_1, \dotsc, v_1 + a e_n \}$, and notice that $\< N
  \> = W$.  We are going to prove that there is at least one vector $u
  \in N$ such that $\chi|_{\< u \>^\triangleright}$ is not symmetric.
  Then we are done by applying the induction hypothesis on $\chi|_{\<
    u \>^\triangleright}$.
	
  Suppose by contradiction that $\chi|_{\< u \>^\triangleright}$ is
  symmetric for every $u \in N$.  In other words, the alternating form
  $\gamma(v, w) := \chi(v, w) - \chi(w, v)$ vanishes on the hyperplane
  $\< u \>^\triangleright$ for every $u \in N$.  In particular, the
  rank of $\gamma$ is at most $2$.  However, the rank of $\gamma$ is
  even (because $\gamma$ is alternating) and non-zero (because $\chi$
  is not symmetric), so it is equal to $2$.  For $u \in W$, denote by
  $\alpha_u, \alpha'_u \in W^*$ the linear forms defined by
  $\alpha_u(w) = \chi(u, w)$ and $\alpha_u'(w) = \gamma(u,w)$.  Let
  $\phi, \psi \colon W \to W^*$ be the linear maps given by $\phi(u) =
  \alpha_u$ and $\psi(u) = \alpha_u'$.  Note that $\phi$ is a vector
  space isomorphism because $\chi$ is non-degenerate, whereas $\psi$
  has rank $2$ because $\gamma$ has rank $2$.  For every $u \in N$
  we have $\gamma|_{\< u \>^\triangleright} = 0$, which can be written
  as: $w \in \ker \alpha'_v$ for every $v, w \in \< u
  \>^\triangleright$.  By definition of $\alpha_u$, we have $\< u
  \>^\triangleright = \ker\alpha_u$.  Therefore, for every $u \in N$
  and $v \in \ker\alpha_u$, we have $\ker \alpha_u \subseteq \ker
  \alpha'_v$ and thus $\alpha'_v$ is a scalar multiple of $\alpha_u$.
  This means that, for every $u \in N$, the image of the restriction
  of $\psi$ to the hyperplane $\ker \alpha_u$ is contained in the
  $1$-dimensional subspace $\< \alpha_u \>$.  Since $\psi$ has rank
  $2$, $\alpha_u$ must be in the image of $\psi$.  Then the isomorphism
  $\phi$ sends $N$ inside the image of $\psi$, which is a
  $2$-dimensional subspace of $V^*$.  This is a contradiction, because
  $N$ spans $W$, whereas the image of $\psi$ has codimension $m-2 \geq
  1$ in $W^*$.
\end{proof}

We are now ready to compute the positive reflection length of any
positive isometry.

\begin{theorem}[Positive reflection length]\label{thm:positive-factorizations}
  Let $(V, Q)$ be a non-de\-ge\-ne\-ra\-te quadratic space over a square-dense
  field $\F$.  Assume that $V$ is not negative definite, and let $f
  \in O_+(V)$ be a positive isometry with $f \neq \id$.  If at least
  one of the following conditions holds:
  \begin{enumerate}[(i)]
  \item $\Mov(f)$ is positive definite,
  \item $f$ is not an involution and $\Mov(f)$ is not negative semi-definite,
  \end{enumerate}
  then the positive reflection length of $f$ is equal to $\dim\Mov(f)$.
  Otherwise, it is equal to $\dim\Mov(f) + 2$.
  In particular, every positive isometry is a product of positive
  reflections.
\end{theorem}

\begin{proof}
  Let $m = \dim\Mov(f) \geq 1$.  If (i) holds, then $\Mov(f)$ is not
  totally singular and $f$ has a direct factorization as a product of
  reflections by \Cref{thm:reflection-length}. These reflections are
  positive, because $\Mov(f)$ is positive definite.

  If (ii) holds, then $\chi_f$ is not symmetric by property (v) of
  \Cref{lemma:wall-form-properties}, and \Cref{lemma:positive-basis}
  yields a basis $e_1, \dotsc, e_m$ such that $\chi_f(e_i, e_i) > 0$
  for all $i$ and $\chi(e_i, e_j) = 0$ for $i < j$.  By
  \Cref{thm:factorization}, we have $f = r_1 \dotsm r_m$ where $r_i$
  is the reflection with respect to $e_i$.  Therefore, $f$ is a product
  of $m$ positive reflections.
	
  Conversely, if $f$ can be written as a product of $m$ positive
  reflections with respect to some positive vectors $e_1, \dotsc,
  e_m$, then by \Cref{thm:factorization} we have $\chi(e_i, e_i) > 0$
  for all $i$ and $\chi(e_i, e_j) = 0$ for $i < j$.  In particular,
  $\Mov(f)$ contains at least one positive vector.  If $\chi_f$ is
  symmetric, then $\Mov(f)$ is positive definite and (i) holds.  If
  $\chi_f$ is not symmetric, then (ii) holds.  Therefore, if both (i)
  and (ii) do not hold, then every factorization of $f$ as a product of
  positive reflections requires at least $m+2$ reflections.
	
  Finally, we are going to show that any positive isometry $f$ can be
  written as a product of $\leq m+2$ positive reflections.  We do this
  by induction on $m$, the case $m=0$ being trivial.  Let $m \geq 1$.
  If $\Mov(f)$ contains at least one positive vector $u$, then we can
  write $f = r_uf'$ where $\dim \Mov(f') = m-1$ by
  \Cref{lemma:reflection-multiplication}, and proceed by induction.
  Therefore we may assume that $\Mov(f)$ is negative semi-definite.
  We are going to show that there is at least one positive vector $v
  \in V$ such that $\chi_{r_vf}$ is not symmetric.  Notice that
  $\Mov(r_vf) = \Mov(f) \oplus \< v \>$ by
  \Cref{lemma:reflection-multiplication}, so $\Mov(r_vf)$ contains the
  positive vector $v$.  Then \Cref{lemma:positive-basis} can be
  applied to $\chi = \chi_{r_vf}$, yielding a factorization of $r_vf$
  as a product of $m+1$ positive reflections, and thus allowing us to
  write $f$ as a product of $m+2$ positive reflections.
	
  We only need to show that, if $\Mov(f) \neq \{0\}$ is negative
  semi-definite, then there is at least one positive vector $v \in V$ such
  that $\chi_{r_vf}$ is not symmetric.  Let $v$ be any positive
  vector.  Recall that $\Mov(f) = \< v \>^\triangleright$, where the
  right orthogonal complement is taken in $\Mov(r_vf) = \Mov(f) \oplus
  \< v \>$ with respect to the bilinear form $\chi_{r_vf}$.  If
  $\chi_{r_vf}$ is symmetric, then $\Mov(r_v f) = \Mov(f) \perp \< v
  \>$.
  Therefore $v \in \Mov(f)^\perp = \Fix(f)$.
  The set of positive vectors of $V$ is non-empty because $V$ is not
  negative definite, and it spans $V$ by
  \Cref{lemma:perturb-positive-vector}.  If $\chi_{r_v f}$ is
  symmetric for all positive vectors $v \in V$, then $v \in \Fix(f)$ for
  all positive vectors $v$, so $\Fix(f) = V$ and thus $f = \id$, which is a
  contradiction.
\end{proof}

We say that an isometry $f \in O_+(V)$ is \emph{positive-minimal} if
it is a product of $\dim\Mov(f)$ positive reflections.
\Cref{thm:positive-factorizations} provides a characterization of
positive-minimal isometries: an involution is positive-minimal if and
only if its moved space is positive definite; a non-involution is
positive-minimal if and only if its moved space is not negative
semi-definite (i.e., it contains at least one positive vector).

If we replace reflection factorizations with \emph{positive}
reflection factorizations in \Cref{def:partial-order}, we obtain a partial order on the group $O_+(V)$.  This is not simply the
restriction to $O_+(V)$ of the partial order on $O(V)$.  Indeed, if $f
\in O_+(V)$ is minimal but not positive-minimal, then there is a
minimal positive factorization $f = r_1 r_2 g$ with $l(g) = l(f) =
\dim\Mov(f)$, and we have $g \leq f$ in $O_+(V)$ but $g \not\leq f$ in
$O(V)$.  For the same reason, the rank function of $O_+(V)$ is not the
restriction of the rank function of $O(V)$.

If $f \in O_+(V)$ is a positive-minimal isometry, then
\Cref{thm:positive-factorizations} allows us to include the interval $[\id,
  f]$ in $O_+(V)$ into the poset of linear subspaces of $\Mov(f)$, in
the same spirit as \Cref{thm:intervals}.

\section{Isometries of the hyperbolic space}
\label{sec:hyperbolic-space}

In this section, we describe reflection length and intervals in the
isometry group of the hyperbolic space $\H^n$.  We follow the notation
of \cite{cannon1997hyperbolic}.

Let $V = \R^{n+1}$, with the quadratic form $Q(x) = x_1^2 +
\dotsb + x_n^2 - x_{n+1}^2$.  Then $(V, Q)$ is a real quadratic space
of signature $(n, 1)$.  The \textit{hyperboloid model} of the hyperbolic space is
\[ \H^n = \{ x \in V \mid Q(x) = -1 \text{ and } x_{n+1} > 0 \}. \]
The quadratic form $Q$
induces a (positive definite) Riemannian metric on $\H^n$.  The
condition $x_{n+1} > 0$ selects the upper sheet of the hyperboloid $\{
Q(x) = -1 \}$.  Every isometry of $\H^n$ uniquely extends to an
isometry of $(V, Q)$; conversely, every isometry of $(V, Q)$ that
fixes $\H^n$ (as a set) restricts to an isometry of $\H^n$.

\begin{lemma}
  The subgroup of $O(V)$ that fixes $\H^n$ (as a set) coincides with
  the index-two subgroup $O_+(V)$ of the positive isometries.
\end{lemma}

\begin{proof}
  Both subgroups have index $2$, so it is enough to show one
  containment.  By \Cref{thm:positive-factorizations}, the subgroup
  $O_+(V)$ is generated by the positive reflections $r \in O(V)$, and
  therefore it is enough to show that every positive reflections fixes
  $\H^n$.  If $v \in V$ is a positive vector, then $\< v \>^\perp$ has
  signature $(n-1, 1)$, so it intersects $\H^n$.  Therefore $r_v$
  fixes at least one point of $\H^n$, so it fixes $\H^n$ as a set.
\end{proof}

Reflections in the hyperbolic space $\H^n$ are restrictions of
positive reflections of $(V, Q)$.  Therefore, the study of reflection
length and intervals in the isometry group of $\H^n$ reduces to the
study of positive reflection length and intervals in $O_+(V)$.  This
is exactly the setting of \Cref{sec:positive-factorizations}.
It turns out that every isometry of $\H^n$ is positive-minimal.

\begin{theorem}\label{thm:hyperbolic-reflection-length}
  The positive reflection length of an isometry $f \in O_+(V)$ is
  equal to $\dim\Mov(f)$.
\end{theorem}

\begin{proof}
  We prove this by induction on $k = \dim\Mov(f)$, the case $k=0$ (the
  identity) being trivial.  If $k = 1$, then $f$ is a positive
  reflection.  If $k \geq 2$, then $\Mov(f)$ intersects the hyperplane
  $\{ x_{n+1} = 0 \}$ non-trivially, so it contains at least one
  positive vector $v$.  By \Cref{thm:factorization}, there is a direct
  factorization $f = r_v g$.  Then $\dim\Mov(g) = k-1$, and $g$ can be
  written as a product of $k-1$ positive reflections by induction.
\end{proof}

We are then able to obtain a clean description of all intervals $[\id,
  f]$ in $O_+(V)$.

\begin{theorem}\label{thm:hyperbolic-space-intervals}
  Let $f \in O_+(V)$.  The interval $[\id, f]$ in $O_+(V)$ is
  isomorphic to the poset of linear subspaces $U \subseteq \Mov(f)$
  such that $\det(\chi_f|_U) > 0$.
\end{theorem}

\begin{proof}
  By \Cref{thm:hyperbolic-reflection-length}, we have that $f$ is
  positive-minimal.  Therefore, all minimal length factorizations of
  $f$ into positive reflections are direct factorizations.  In
  particular, the interval $[\id, f]$ in $O_+(V)$ is contained in the
  interval $[\id, f]$ in the whole group $O(V)$.  To avoid confusion,
  denote by $[\id, f]_+$ the interval in $O_+(V)$.  If $g \in [\id,
    f]$ is a positive isometry, then $h = g^{-1}f$ is also positive,
  and $g$ and $h$ are positive-minimal by
  \Cref{thm:hyperbolic-reflection-length}.  Therefore $g \in [\id,
    f]_+$.  This shows that $[\id, f]_+ = [\id, f] \cap O_+(V)$.
	
  By \Cref{thm:intervals}, the map $g \mapsto \Mov(g)$ is a bijection
  between $[\id, f]_+$ and the poset of linear subspaces $U \subseteq
  \Mov(f)$ such that: $U$ satisfies conditions (i)-(iii) of
  \Cref{thm:intervals}; (iv) $\det(\chi_f|_U) > 0$ (this is the same
  as saying that the preimage of $U$ is a positive isometry).  Since
  the signature of $V$ is $(n, 1)$, the totally singular subspaces
  have dimension $0$ or $1$, so conditions (i) and (ii) are implied by
  condition (iii).  In addition, we can disregard condition (iii) as
  it is implied by (iv).  Putting everything together, the map $g
  \mapsto \Mov(g)$ is a bijection between $[\id, f]_+$ and the poset of
  linear subspaces $U \subseteq \Mov(f)$ satisfying $\det(\chi_f|_U) >
  0$.
	
  If $g \leq g'$ in $[\id, f]_+$, then $g \leq g'$ in $[\id, f]$, and
  thus $\Mov(g) \subseteq \Mov(g')$ by \Cref{thm:intervals}.
  Conversely, suppose that we have $g, g' \in [\id, f]_+$ such that $\Mov(g)
  \subseteq \Mov(g')$.
  By \Cref{lemma:isometries-below-f}, $\chi_g$ and $\chi_{g'}$ are the restrictions of $\chi_f$ to $\Mov(g)$ and $\Mov(g')$, respectively.
  Then $\chi_g = \chi_{g'}|_{\Mov(g)}$, so there is a direct factorization
  $g' = gh$ and $h$ is positive-minimal by
  \Cref{thm:hyperbolic-reflection-length}.  Therefore $g \leq g'$ in
  $[\id, f]_+$.  This shows that the bijection $g \mapsto \Mov(g)$ is
  a poset isomorphism.
\end{proof}

Notice that \Cref{thm:hyperbolic-space-intervals} gives a poset
isomorphism, whereas \Cref{thm:intervals} only gives an
order-preserving bijection.
A counterexample like the one in
\Cref{example:direct-factorization} cannot occur in this context,
since all positive isometries are positive-minimal.
Indeed, for \Cref{example:direct-factorization} to arise, the Witt index of the ambient space $V$ needs to be at least $2$ (in other words, over an ordered field, the signature needs to be $(p, q)$ with $p, q \geq 2$).

It is also true that all isometries of $O(V)$ are minimal, by
\Cref{thm:reflection-length}.  Indeed, the only non-trivial totally
singular subspaces are one-dimensional, and they do not arise as moved
spaces of any isometry, because the Wall form would be identically
zero.

Recall that, if we interpret the hyperboloid model as lying in the projective space
$\P(V)$, the singular lines $\< v \> \subseteq \{ Q(x) = 0 \}$ can be
interpreted as ``points at infinity'' of the hyperbolic space $\H^n$.
Then the isometries of $\H^n$ can be classified into three types:
\emph{elliptic} isometries, that fix at least one point of $\H^n$;
\emph{parabolic} isometries, that fix no point of $\H^n$ and fix
exactly one point at infinity; \emph{hyperbolic} isometries, that fix
no point of $\H^n$ and fix two points at infinity.  See \cite[Section
  12]{cannon1997hyperbolic}.  We now rewrite this classification in
terms of fixed space and moved space.

\begin{definition}
  An isometry $f \in O_+(V)$ is
  \begin{itemize}
  \item \emph{elliptic} if $\Fix(f)$ contains a negative vector
    (i.e., it is not positive semi-definite);
  \item \emph{parabolic} if $\Fix(f)$ is positive semi-definite but
    not positive definite;
  \item \emph{hyperbolic} if $\Fix(f)$ is positive definite.
  \end{itemize}
\end{definition}

\begin{lemma}\label{lemma:hyperbolic-isometries-fix-mov}
  Let $f \in O_+(V)$.  We have that $\Fix(f) \cap
  \Mov(f) = \{0\}$ if $f$ is elliptic or hyperbolic, whereas $\Fix(f)
  \cap \Mov(f)$ is a singular line if $f$ is parabolic.  In
  addition:
  \begin{itemize}
  \item $f$ is elliptic if and only if $\Mov(f)$ is positive definite; 
  \item $f$ is parabolic if and only if $\Mov(f)$ is positive
    semi-definite but not positive definite; 
  \item $f$ is hyperbolic if and only if $\Mov(f)$ contains a negative
    vector. 
  \end{itemize}
\end{lemma}

\begin{proof}
  We have that $\Mov(f) = \Fix(f)^\perp$ by
  \Cref{lemma:fix-move-orthogonal}.  Therefore $\Fix(f) \cap \Mov(f)$
  is a totally singular subspace, so its dimension is at most $1$.
  If $\Fix(f) \cap \Mov(f)$ contains a non-trivial singular vector
  $v$, then $\Fix(f)$ is not positive definite, so $f$ is elliptic or
  parabolic.
	
  If $f$ is elliptic, then up to conjugating by an isometry in
  $O_+(V)$ we may assume that $f$ fixes the point $e_{n+1} = (0,
  \dotsc, 0, 1) \in \H^n$.  Then $f$ is an isometry also with respect
  to the standard (positive definite) Euclidean quadratic form $Q_E(x)
  = x_1^2 + \dotsc + x_{n+1}^2$.  Therefore $\Fix(f)$ and $\Mov(f)$
  are $Q_E$-orthogonal by \Cref{lemma:fix-move-orthogonal}, and in
  particular $\Fix(f) \cap \Mov(f) = \{0 \}$.
  If $f$ is parabolic, then $\Fix(f)$ contains a singular line, so $\Fix(f) \cap \Mov(f)$ is a singular line.
  This finishes the proof of the first part of the statement.
	
  We now prove the classification in terms of the moved space.  If $f$
  is elliptic, then $\Fix(f)$ contains a negative vector and $V =
  \Fix(f) \perp \Mov(f)$, so $\Mov(f)$ is positive definite.
  Similarly, if $f$ is hyperbolic, then $\Fix(f)$ is positive definite
  and $V = \Fix(f) \perp \Mov(f)$, so $\Mov(f)$ contains a negative
  vector.  If $f$ is parabolic, then $\Mov(f)$ contains a singular
  vector and so it is not positive definite.  Finally, if $\Mov(f)$
  contains a negative vector $w$, then $\< w \>^\perp$ is positive
  definite and $\Fix(f) = \Mov(f)^\perp \subseteq \< w \>^\perp$, so
  $f$ is not parabolic.
\end{proof}

For elliptic isometries, the description of the intervals given by
\Cref{thm:hyperbolic-space-intervals} becomes particularly simple
thanks to the following observation.

\begin{lemma}\label{lemma:wall-form-positive-definite}
  Let $f \in O_+(V)$.  If $U \subseteq \Mov(f)$
  is a positive definite subspace, then $\det(\chi_f|_{U}) > 0$.
\end{lemma}

\begin{proof}
  The restriction $\chi_f|_U$ is non-degenerate, because $\chi(u, u) =
  Q(u) > 0$ for all $u \in U$.  Applying \Cref{lemma:triangular-basis}
  to $\chi_f|_U$, we obtain a basis $e_1, \dotsc, e_m$ of $U$ such
  that $\chi_f(e_i, e_i) \neq 0$ for all $i$, and $\chi(e_i, e_j) = 0$
  for $i < j$.  Additionally, we have $\chi_f(e_i, e_i) = Q(e_i) > 0$
  for all $i$.  Therefore, $\det(\chi_f|_U) > 0$.
\end{proof}

\begin{theorem}[Elliptic intervals]
  Let $f \in O_+(V)$ be an elliptic isometry.  Then the interval
  $[\id, f]$ is isomorphic to the poset of all linear subspaces of
  $\Mov(f)$.  In particular, the isomorphism type of $[\id, f]$ only
  depends on the dimension of $\Mov(f)$, and not on the Wall form
  $\chi_f$.
\end{theorem}

\begin{proof}
  This follows immediately from \Cref{thm:hyperbolic-space-intervals}
  and \Cref{lemma:wall-form-positive-definite}.
\end{proof}

The description of \Cref{thm:hyperbolic-space-intervals} can be
simplified also for parabolic intervals.

\begin{lemma}\label{lemma:wall-form-restriction-degenerate}
  Let $f \in O_+(V)$ be a positive isometry, and $U \subseteq \Mov(f)$
  a subspace.  The restriction $\chi_f|_U$ is degenerate if and only
  if there is a singular vector $v \in \Mov(f) \setminus \{0\}$ such
  that $\< v \> \subseteq U \subseteq \< v \>^\triangleright$.
  Note that $\< v \>^\triangleright = \< w \>^\perp$ where $w$ is any vector such that $w - f(w) = v$.
\end{lemma}

\begin{proof}
  The restriction $\chi_f|_U$ is degenerate if and only if there is a
  non-zero vector $v \in U$ such that $\chi_f(v, u) = 0$ for all $u
  \in U$, or equivalently $\< v \> \subseteq U \subseteq \< v
  \>^\triangleright$.  Since $\chi_f(v, v) = Q(v)$, the containment
  $\< v \> \subseteq \< v \>^\triangleright$ holds if and only if $v$
  is singular.
  Finally, by definition of $\chi_f$, we have $\chi_f(v, u) = \beta(w, u)$ for all $u \in U$, and therefore $\< v \>^\triangleright = \< w \>^\perp$.
\end{proof}

\begin{theorem}[Parabolic intervals]\label{thm:parabolic-intervals}
  Let $f \in O_+(V)$ be a parabolic isometry which pointwise fixes
  the singular line $\< v \>$.  Then the interval $[\id, f]$ is
  isomorphic to the poset of linear subspaces $U \subseteq \Mov(f)$
  that do not satisfy $\< v \> \subseteq U \subseteq \< v
  \>^\triangleright$.  In particular, the isomorphism type of $[\id,
    f]$ only depends on the dimension of $\Mov(f)$, and not on the
  Wall form $\chi_f$.
\end{theorem}

\begin{proof}
  Let $U \subseteq \Mov(f)$ be a subspace.  If $\< v \> \nsubseteq U$,
  then $U$ is positive definite and thus $\det(\chi_f|_U) > 0$ by
  \Cref{lemma:wall-form-positive-definite}.  Since $\< v \>$ is the
  only singular line in $\Mov(f)$, the restriction $\chi_f|_U$ is
  degenerate if and only if $\<v \> \subseteq U \subseteq \< v
  \>^\triangleright$ by
  \Cref{lemma:wall-form-restriction-degenerate}.  Finally, if $\< v \>
  \subseteq U \nsubseteq \< v \>^\triangleright$, then \Cref{lemma:triangular-basis} yields a basis
  $e_1, \dotsc, e_m$ of $U$ such that $\chi_f(e_i, e_i) \neq 0$ for
  all $i$ and $\chi_f(e_i, e_j) = 0$ for $i < j$.  Since $f$ is
  parabolic, $\Mov(f)$ is positive semi-definite by
  \Cref{lemma:hyperbolic-isometries-fix-mov} and therefore $\chi(e_i,
  e_i) = Q(e_i) > 0$ for all $i$. Thus $\det(\chi_f|_U) > 0$ also in
  this case.  We conclude by applying
  \Cref{thm:hyperbolic-space-intervals}.
\end{proof}

The subgroup of $O_+(V)$ that fixes a singular line $\< v \>$ is
isomorphic to the isometry group of the affine Euclidean space
$\R^n$. This is easily seen in the \emph{half-space model} of the
hyperbolic space (see \cite[Section 12]{cannon1997hyperbolic}).  In
particular, parabolic intervals are isomorphic to intervals in the
group of affine Euclidean isometries, which have been explicitly
described in \cite{brady2015factoring}.
Our description is more compact than the one of \cite{brady2015factoring}, where the elliptic and the parabolic portions of an interval are described separately.

The results of this section leave open the following natural question: if $f \in O_+(V)$ is a hyperbolic isometry, does the isomorphism type of $[\id, f]$ depend only on the dimension of $\Mov(f)$?

\bibliographystyle{amsalpha-abbr}
\bibliography{bibliography}

\end{document}